\documentclass[reqno]{amsart}
\usepackage{amsfonts, amsmath, amssymb, amsthm, color}
\usepackage[margin=2.75cm, heightrounded]{geometry}
\usepackage{nccmath}

\newtheorem{theorem}{Theorem}[section]

\newtheorem{proposition}[theorem]{Proposition}
\newtheorem{lemma}[theorem]{Lemma}
\newtheorem{cor}[theorem]{Corollary}
\theoremstyle{remark}
\newtheorem{remark}[theorem]{Remark}

\newcommand{\pa}{\partial}
\newcommand{\ep}{\epsilon}
\newcommand{\ga}{\gamma}

\newcommand{\R}{\mathbb{R}}

\newcommand{\mone}{\mathbf{1}}

\newcommand{\mbN}{\mathbf{N}}

\newcommand{\mca}{\mathcal{A}}
\newcommand{\mce}{\mathcal{E}}
\newcommand{\mci}{\mathcal{I}}
\newcommand{\mcp}{\mathcal{P}}
\newcommand{\mcq}{\mathcal{Q}}
\newcommand{\mcr}{\mathcal{R}}

\newcommand{\tc}{\tilde{c}}
\newcommand{\td}{\tilde{d}}
\newcommand{\tta}{\tilde{\tau}}

\newcommand{\whh}{\widehat{H}}

\newcommand{\wtg}{\widetilde{G}}
\newcommand{\wth}{\widetilde{H}}
\newcommand{\wtj}{\widetilde{J}}
\newcommand{\wtr}{\widetilde{R}}
\newcommand{\ovom}{\overline{\Omega}}

\newcommand{\dx}{\textup{d}x}
\newcommand{\dy}{\textup{d}y}

\newcommand{\oms}{\Omega_*}

\newcommand{\mx}{{\mu,\xi}}

\renewcommand{\(}{\left(}
\renewcommand{\)}{\right)}

\allowdisplaybreaks
\numberwithin{equation}{section}

\begin{document}
\title{Coron's problem for the critical Lane-Emden system}

\author[S. Jin]{Sangdon Jin}
\address[Sangdon Jin]{Department of Mathematics, Chung-Ang University, Seoul 06974, Korea}
\email{sdjin@cau.ac.kr}

\author[S. Kim]{Seunghyeok Kim}
\address[Seunghyeok Kim]{Department of Mathematics and Research Institute for Natural Sciences, College of Natural Sciences, Hanyang University,
222 Wangsimni-ro Seongdong-gu, Seoul 04763, Republic of Korea}
\email{shkim0401@hanyang.ac.kr shkim0401@gmail.com}

\begin{abstract}
In this paper, we address the solvability of the critical Lane-Emden system
\[\begin{cases}
-\Delta u=|v|^{p-1}v &\textup{in } \Omega_\ep,\\
-\Delta v=|u|^{q-1}u &\textup{in } \Omega_\ep,\\
u=v=0 &\textup{on } \pa \Omega_\ep
\end{cases}\]
where $N \ge 4$, $p \in (1,\frac{N-1}{N-2})$, $\frac{1}{p+1} + \frac{1}{q+1}=\frac{N-2}{N}$, and $\Omega_\ep$ is a smooth bounded domain with a small hole of radius $\ep > 0$.
We prove that the system admits a family of positive solutions that concentrate around the center of the hole as $\ep \to 0$, obtaining a concrete qualitative description of the solutions as well.
To the best of our knowledge, this is the first existence result for the critical Lane-Emden system on a bounded domain,
while the non-existence result on star-shaped bounded domains has been known since the early 1990s due to Mitidieri (1993) \cite{Mi} and van der Vorst (1991) \cite{Va}.
\end{abstract}

\date{\today}
\maketitle

\section{Introduction}
Let $N \ge 3$, $p > 1$, and $\Omega$ be a smooth domain that is either bounded or the whole space $\R^N$. More than half-century, the Lane-Emden equation
\begin{equation}\label{LE eq0}
-\Delta u = |u|^{p-1}u \quad \textup{in } \Omega
\end{equation}
has played a role as one of the fundamental equations in the theory of nonlinear partial differential equations, thanks to its profound structural complexity despite its simple appearance.
As is well-known, the order relation between $p$ and the Sobolev exponent $p_S := \frac{N+2}{N-2}$ influences the solution structure of \eqref{LE eq0} in a significant way.
Moreover, the solution structure of the critical Lane-Emden equation (i.e., \eqref{LE eq0} with $p = p_S$)
\begin{equation}\label{LE eq}
\begin{cases}
-\Delta u = |u|^{4 \over N-2}u &\textup{in } \Omega,\\
u \in H^1_0(\Omega)
\end{cases}
\end{equation}
subtly depends on the topology and geometry of the domain $\Omega$.

In this paper, we are interested in an elliptic system
\begin{equation}\label{main system0}
\begin{cases}
-\Delta u=|v|^{p-1}v \quad \textup{in } \Omega,\\
-\Delta v=|u|^{q-1}u \quad \textup{in } \Omega,\\
(u,v) \in X_{p,q}(\Omega),
\end{cases}
\end{equation}
called the Lane-Emden system. Here, $N \ge 3$, $(p,q)$ is a pair of positive numbers, $\Omega$ is either a smooth bounded domain or the whole space $\R^N$, and $X_{p,q}(\Omega)$ is a natural energy space compatible to the Dirichlet boundary condition (see \eqref{X pq}).
For this, the Sobolev hyperbola
\begin{equation}\label{critical}
\frac{1}{p+1}+\frac{1}{q+1} = \frac{N-2}{N}
\end{equation}
plays a similar role to the Sobolev exponent $p_S$ for equation \eqref{LE eq0}, and the critical Lane-Emden system \eqref{main system0}--\eqref{critical} serves as the natural counterpart of \eqref{LE eq} for Hamiltonian-type elliptic systems;
refer to a survey paper \cite{BMT} for detailed account of Hamiltonian elliptic systems.

System \eqref{main system0} has received remarkable attention for decades.
When $\Omega = \R^N$, Lions \cite{Li} found a positive least energy solution to \eqref{main system0}--\eqref{critical}
in $X_{p,q}(\R^N) = \dot{W}^{2,\frac{p+1}{p}}(\R^N) \times \dot{W}^{2,\frac{q+1}{q}}(\R^N)$ by using the concentration-compactness principle.\footnote{We remark
that the energy functional (defined in \eqref{ene}) of system \eqref{main system0}--\eqref{critical} is strongly indefinite in $X_{p,q}(\R^N)$,
and the least energy solutions have the Morse index $+\infty$.}
Alvino et al. \cite{ALT} then showed that it is radially symmetric and decreasing in the radial variable, after a suitable translation.
After that, Wang \cite{Wa} and Hulshof and van der Vorst \cite{HV} independently deduced that it is unique up to translations and dilations.
Under the further assumption that $p \ge 1$, Chen et al. \cite{CLO} extended their results by showing that a positive solution to \eqref{main system} is unique up to translations and dilations.
Recently, Frank, Pistoia, and the second author of this paper \cite{FKP} established that all least energy solutions are non-degenerate in the sense that their linearized equations have precisely the $(N+1)$-dimensional space
of solutions in $X_{p,q}(\R^N)$, originating from the translation and dilation invariance of \eqref{main system0}.
Furthermore, by employing the variational method, Clapp and Salda\~na \cite{CS} found finitely many non-radial sign-changing solutions to \eqref{main system0}--\eqref{critical} for $N \ge 4$.

On the other hand, there is a longstanding conjecture in the literature, called the Lane-Emden conjecture. It claims that if the pair $(p,q)$ satisfies
\begin{equation}\label{subcritical}
\frac{1}{p+1} + \frac{1}{q+1} > \frac{N-2}{N},
\end{equation}
then \eqref{main system0} with $\Omega = \R^N$ admits no positive solutions. The complete answer is unknown yet, though partial results are available in, e.g., \cite{So, CH} and references therein.

If $\Omega$ is a smooth bounded domain, then the situation changes drastically. Due to the works of Hulshof and van der Vorst \cite{HV2}, Figueiredo and Felmer \cite{FF},
and Bonheure et al. \cite{BMR}, it is known that \eqref{main system0} with \eqref{subcritical} has a solution provided $pq \ne 1$.
Conversely, according to Mitidieri \cite{Mi} and van der Vorst \cite{Va}, if $\Omega$ is star-shaped, then \eqref{main system0}--\eqref{critical} has no positive solution.
If $pq > 1$, then the standard variational method produces a positive least energy solution to \eqref{main system0} with \eqref{subcritical}.
Guerra \cite{Gu} and Choi and the second author of this paper \cite{CK} investigated the asymptotic behavior of the least energy solutions as $(p,q)$ approaches the Sobolev hyperbola \eqref{critical},
showing that the least energy solution blows up at a certain point in $\Omega$.
By applying the perturbative argument, the second author and Pistoia \cite{KP} built multi-bubble (so higher energy) solutions to \eqref{main system0}
when $(p,q)$ satisfies \eqref{subcritical} and is sufficiently close to \eqref{critical}, and $\Omega$ is dumbbell-shaped.
Also, when $\Omega$ is convex, the second author and Moon \cite{KM} classified the asymptotic behavior of all positive solutions to \eqref{main system0} with \eqref{subcritical} as $(p,q)$ tends to \eqref{critical}.

However, as far as we know, no existence result for the critical system \eqref{main system0}--\eqref{critical} is known in the literature.\footnote{
The existence theory of linear perturbations of the critical Lane-Emden system \eqref{main system0}--\eqref{critical} in a smooth bounded domain $\Omega$,
i.e., the Brezis-Nirenberg type problem was studied in \cite{HMV, Gu, CK, KP, GHP}.}
Therefore, it is natural to ask if the system can possess a positive solution in a bounded domain $\Omega$ under certain conditions on the topology or geometry on $\Omega$.
In this paper, we give an affirmative answer for the question by constructing solutions to the system
assuming that $\Omega$ has a sufficiently small hole, motivated by Coron's result \cite{Co} for equation \eqref{LE eq}.
Unlike Coron who used the variational method, we apply a perturbative argument. As a by-product, we obtain a concrete qualitative description of the solutions.

In what follows, we pay attention to
\begin{equation}\label{main system}
\begin{cases}
-\Delta u=|v|^{p-1}v &\textup{in } \Omega_\ep,\\
-\Delta v=|u|^{q-1}u &\textup{in } \Omega_\ep,\\
u=v=0 &\textup{on } \pa \Omega_\ep
\end{cases}
\end{equation}
where $N \ge 4$, $\Omega$ is a smooth bounded domain in $\R^N$ such that $0 \in \Omega$, $\ep>0$ is a small number, $\Omega_\ep : =\Omega\setminus \overline{B(0,\ep)}$, $p \in (1,\frac{N-1}{N-2})$,
and the pair $(p,q)$ is on the Sobolev hyperbola \eqref{critical}.\footnote{The condition $p \in (1,\frac{N-1}{N-2})$ forces us to assume $N \ge 4$. If $N = 3$, we must have $p \in (1,2)$.
On the other hand, \eqref{critical} implies that $\frac{1}{3} - \frac{1}{p+1} = \frac{1}{q+1} > 0$, which implies $p > 2$.}
\begin{theorem}\label{thm main}
There exists a small number $\ep_0 > 0$ such that system \eqref{main system} has a positive solution $(u_\ep,v_\ep) \in (C^2(\ovom_\ep))^2$ for each $\ep \in (0,\ep_0)$
Moreover, the family $\{(u_\ep,v_\ep)\}_{\ep \in (0,\ep_0)}$ concentrate around the origin of $\R^N$ as $\ep \to 0$ in the sense that
\begin{equation}\label{eq main}
(u_\ep,v_\ep) = \(\mcp U_\mx+\psi_{\ep,(d,\tau)}, PV_\mx+\phi_{\ep,(d,\tau)}\) \quad \textup{in } \Omega_\ep.
\end{equation}
Here, $(\mcp U_\mx,PV_\mx)$ is the main part of the solution $(u_\ep,v_\ep)$ defined by \eqref{PUPV eq} and \eqref{mcp U eq},
\begin{equation}\label{eq main1}
(\mu,\xi) := (\ep^\alpha d, \ep^\alpha d\tau) \in (0,\infty) \times \R^N \quad \textup{where } \alpha := \frac{N-2}{(N-2)p+N-4} \in (0,1)
\end{equation}
for some suitably chosen $d > 0$ and $\tau \in \R^N$, and
\[\|\Delta \psi_{\ep,(d,\tau)}\|_{L^{\frac{p+1}{p}}(\Omega_\ep)} + \|\Delta \phi_{\ep,(d,\tau)}\|_{L^{\frac{q+1}{q}}(\Omega_\ep)} \to 0 \quad \textup{as } \ep \to 0.\]
\end{theorem}
\begin{remark}\label{main rmk}
We make some comments on the above theorem.

\medskip \noindent (1) In the proof of Theorem \ref{thm main}, we will find several interesting and unique features of system \eqref{main system}. They include
\begin{itemize}
\item[-] the extremely slow decay of the $u$-component $U_\mx$ of a positive solution $(U_\mx,V_\mx)$ in \eqref{UVmx} to the limit system \eqref{limit system} (see Lemma \ref{limit decay});
\item[-] a strong nonlinear behavior represented by the nonlinear projection $\mcp U_\mx$ of $U_\mx$ in \eqref{mcp U eq} and the potential analysis in the proof of Propositions \ref{second approx} and \ref{second approx2};
\item[-] entanglement between the components $u$ and $v$ of solutions $(u,v)$ to \eqref{main system}, which can be seen in the expansion \eqref{mcp U} of $\mcp U_\mx$ and the definition of $\mca_{\ep,(d,\tau)}$ in \eqref{A mu}
\end{itemize}
among others.

\medskip \noindent (2) We could treat the range $[\frac{N-1}{N-2},\frac{N+2}{N-2}]$ of $p$ as well, but opted to focus on $p \in (1,\frac{N-1}{N-2})$ for the following reasons:
If $p \in [\frac{N}{N-2},\frac{N+2}{N-2}]$, system \eqref{main system} acts as the scalar equation \eqref{LE eq}, which corresponds to $p=\frac{N+2}{N-2}$ for our system. Hence accompanying analysis is expected to be rather traditional.\footnote{Nevertheless,
we still need some additional work to handle when $p$ is close to $\frac{N}{N-2}$. We ask the interested reader to consult \cite[Subsections 5.1, 5.2]{KM}.}
If $p \in [\frac{N-1}{N-2},\frac{N}{N-2})$, the system behaves similarly to our case, but an additional technical issue arises involving the regularity of the auxiliary map $\wth_y$ in \eqref{wth}.
It makes the problem extremely complicated and, in our opinion, somewhat hides the features of the system depicted above.

A suitable combination of the argument in this paper and one in \cite[Section 5]{KM} may produce the existence result analogous to Theorem \ref{thm main} for all $p \in (1,\frac{N+2}{N-2}]$.

If $p = 1$, the system boils down to the biharmonic Lane-Emden equation studied in \cite{AP}.
In our proof, the condition $p > 1$ is crucially used in several places - the setting of the linear theory, the energy expansion, and the regularity of the solutions.
It seems difficult to deal with the case $p \in (\frac{2}{N-2},1)$ with our argument.

\medskip \noindent (3) We see that the exponent $\alpha$ in \eqref{eq main1} tends to $\frac{N-2}{2(N-3)}$ as $p \to 1^+$, which matches the exponent in \cite[(2.3)]{AP}.
Also, a formal computation reveals that the blow-up rate of possible single-bubble solutions to \eqref{main system} must be $\ep^{\alpha}$
for $p \in [\frac{N-1}{N-2},\frac{N}{N-2})$, and $\ep^{\frac{1}{2}}$ for $p \in (\frac{N}{N-2},\frac{N+2}{N-2}]$.
Our computation for $p = \frac{N+2}{N-2}$ is consistent to \cite[Theorem 1.1]{LYY2}. Also, $\alpha \to \frac{1}{2}$ as $p \to (\frac{N}{N-2})^-$.
\end{remark}

There have been extensive studies on the Coron-type problems, namely, the existence theory of solutions to a critical elliptic equation or system in a smooth bounded domain with single or multiple holes.

The Pohozaev identity tells us that \eqref{LE eq} has no positive solution if $\Omega$ is star-shaped.
In contrast, if $\Omega$ is an annulus, \eqref{LE eq} has a positive radial solution, as shown by Kazdan and Warner \cite{KW}.
Later, Coron \cite{Co} revealed that \eqref{LE eq} has a positive solution whenever the (possibly non-symmetric) smooth bounded domain $\Omega$ has a sufficiently small hole.
Lewandowski \cite{Le} proved that if the hole is spherical, say $B(0,\ep) := \{x \in \R^N: |x| < \ep\}$, Coron's solutions concentrate around the origin of $\R^N$ as $\ep \to 0$.
Coron's result was substantially improved by the work \cite{BC} of Bahri and Coron, which yields that
if the singular homology group $H_d(\Omega;\mathbb{Z}_2)$ with coefficients in $\mathbb{Z}_2$ is non-trivial for some $d \in \mathbb{N}$, then \eqref{LE eq} has a positive solution.

Applying the perturbative argument, researchers attempted to find more solutions to \eqref{LE eq}.
Under the presence of one or more spherical holes in $\Omega$, Li et al. \cite{LYY2} proved the existence of single-bubble solutions.
Furthermore, Musso and Pistoia \cite{MP} and Ge et al. \cite{GMP} built sign-changing bubble-tower solutions, where the bubble-tower solution refers to a solution which looks like superpositions of bubbles with different blow-up rates.
For further results, we refer to, e.g., \cite{MP2, CMP, DM}.

In addition, Alarc\'on and Pistoia \cite{AP} obtained a positive single-bubble solution for the biharmonic Lane-Emden equation, which corresponds to the critical Lane-Emden system \eqref{main system} with $p = 1$.
For the fractional Lane-Emden equation, Long et al. \cite{LYY} and Chen et al. \cite{CLY} derived the existence of a positive single-bubble solution and sign-changing bubble-tower solutions, respectively.

Lastly, we point out that Coron-type results are available for the critical coupled Schr\"odinger systems, which are the natural counterparts of \eqref{LE eq} for gradient-type elliptic systems.
Indeed, thanks to the works of Pistoia and Soave \cite{PS} and Pistoia et al. \cite{PST}, the existence of positive solutions whose component looks like a single-bubble or a bubble-tower is known.\footnote{The Lane-Emden system
\eqref{main system0} is strongly coupled, and only synchronized blowing-up solutions (namely, solutions whose components have the same blow-up point) can exist.
On the other hand, the coupled nonlinear Schr\"odinger system is weakly coupled, and both segregated blowing-up solutions (namely, solutions whose components have different blow-up points) and synchronized blowing-up solutions may exist.}

\medskip \noindent \textbf{Structure of the paper and novelty of our proof.}
We apply the Lyapunov-Schmidt reduction method to construct the desired solutions to system \eqref{main system}.
For its successful application, we have to find a very precise ansatz for the solutions that reflects the system's unique characteristics (see Remark \ref{main rmk} (1)) and analyze it carefully, which are the main difficulties in our proof.

In Section \ref{sec pre}, we collect some preliminary results such as the definition of the appropriate function spaces and auxiliary maps,
the expansion of the regular part of Green's function of the Laplacian $-\Delta$ in $\Omega_\ep$, and properties of the solutions to the limit system \eqref{limit system}.

In Section \ref{sec app}, we define and improve ansatz of the solutions to \eqref{main system}.
In particular, we set the $u$-component of the refined ansatz as the nonlinear projection $\mcp U_\mx$ of $U_\mx$ (see \eqref{UVmx} and \eqref{mcp U eq}),
and expand it by introducing a suitable function $\mca_{\ep,(d,\tau)}$ and conducting a delicate potential analysis; refer to Proposition \ref{second approx}.
Because our solutions will be characterized as saddle points of the reduced energy $J_\ep$ in \eqref{red ene}, the $C^1$-smallness of the remainder term $R$ in \eqref{R} is required.
Unlike the corresponding situation for \eqref{LE eq} or other related equations, the presence of the `ugly' function $\mca_{\ep,(d,\tau)}$ makes such estimate non-trivial.
In Proposition \ref{second approx2}, we tackle this issue through a delicate analysis.

In Section \ref{sec red}, we perform the $C^1$-estimate of the reduced energy $J_\ep$.
Lots of technical computations appearing in Sections \ref{sec app} and \ref{sec red} are postponed to Appendix \ref{sec tech}.

In Section \ref{sec LSred}, we carry out the reduction procedure.
Several results in this section easily follow from known arguments. For the conciseness of the paper, we concentrate only on the results that do not. Whenever we omit the details, we will leave appropriate references.

In Section \ref{sec comp}, we complete the proof of Theorem \ref{thm main}.

In Appendix \ref{sec bdd}, we study the regularity of $(\psi_{\ep,(d,\tau)},\phi_{\ep,(d,\tau)})$ (see \eqref{eq main}) by establishing a general regularity result
for linear Hamiltonian-type elliptic systems.\footnote{There is a gap in the proof of \cite[Proposition 4.6]{KP},
since the functions $Q_1$ and $Q_2$ in \cite[(B.5)]{KP} do not belong to $L^{\infty}(\Omega)$. We point out that the argument in Appendix \ref{sec bdd} fills this gap.}

\medskip \noindent \textbf{Notations.}
Here, we list some notations used in the paper.

\medskip \noindent - Given $x \in \R^N$ and $r > 0$, $B(x,r)$ stands for an open ball centered at $x$ of radius $r > 0$.

\medskip \noindent - $\mathbb{S}^{N-1}$ is the unit sphere in $\R^N$ centered at the origin.

\medskip \noindent - Unless otherwise stated, $C > 0$ is a universal constant that may vary from line to line and even in the same line.

\medskip \noindent - Let (A) be a condition. We set $\mone_{\textup{(A)}} = 1$ if (A) holds and $0$ otherwise.

\medskip \noindent - $O(1)$ denotes a term such that
\begin{equation}\label{O(1)}
|O(1)| \le C \quad \textup{for all } \ep > 0 \textup{ small } (\textup{and } x \in \Omega_\ep \textup{ if it is a function in } \Omega_\ep)
\end{equation}
where $C > 0$ is independent of $\ep$ (and $x$). Also, $o(1)$ denotes a term such that
\begin{equation}\label{o(1)}
|o(1)| \le C_{\ep} \quad \textup{for all } \ep > 0 \textup{ small } (\textup{and } x \in \Omega_\ep \textup{ if it is a function in } \Omega_\ep)
\end{equation}
where $0 < C_{\ep} \to 0$ as $\ep \to 0$.

\medskip \noindent \textbf{Caution.} In the following, we always assume that the pair $(p,q)$ satisfies \eqref{critical}.

\section{Preliminaries}\label{sec pre}
\subsection{Functional setting}
Let
\begin{equation}\label{pq star}
\frac{1}{p^\ast} = \frac{p}{p+1} - \frac{1}{N} = \frac{1}{q+1} + \frac{1}{N}
\quad \textup{and} \quad
\frac{1}{q^\ast} = \frac{q}{q+1} - \frac{1}{N} = \frac{1}{p+1} + \frac{1}{N}
\end{equation}
so that $p^\ast$ and $q^\ast$ are the H\"older's conjugates of each other. For any smooth domain $\oms$ in $\R^N$, we set a Banach space
\begin{equation}\label{X pq}
X_{p,q}\big(\oms\big) := \(W^{2,\frac{p+1}{p}}\big(\oms\big) \cap W_0^{1,p^\ast}\big(\oms\big)\) \times \(W^{2,\frac{q+1}{q}}\big(\oms\big) \cap W_0^{1,q^\ast}\big(\oms\big)\)
\end{equation}
equipped with the norm
$$
\|(u,v)\|_{X_{p,q}(\oms)} := \|\Delta u\|_{L^{\frac{p+1}{p}}(\oms)} + \|\Delta v\|_{L^{\frac{q+1}{q}}(\oms)}.
$$
By the Sobolev embedding theorem, we have that $X_{p,q}(\oms) \subset L^{q+1}(\oms) \times L^{p+1}(\oms)$.
For small $\ep > 0$, we also define an energy functional $I_\ep:X_{p,q}(\Omega_\ep) \to \R$ by
\begin{equation}\label{ene}
I_\ep(u,v) = \int_{\Omega_\ep} \(\nabla u\cdot \nabla v-\frac{1}{p+1}|v|^{p+1}-\frac{1}{q+1}|u|^{q+1}\).
\end{equation}
It is of class $C^2$ and its critical point is a weak solution to \eqref{main system}.

\subsection{Green's functions and related maps}
For any smooth domain $\oms$ in $\R^N$, let $G_{\oms}$ be the Green's function of the Laplacian $-\Delta$ in $\oms$ with the Dirichlet boundary condition
and $H_{\oms}:\oms \times \oms \to \R$ its regular part solving
$$
\begin{cases}
-\Delta_x H_{\oms}(x,y)=0 &\textup{for } x \in \oms,\\
H_{\oms}(x,y)=\dfrac{\ga_N}{|x-y|^{N-2}} &\textup{for } x \in \pa \oms
\end{cases}
$$
where $\ga_N := [(N-2)|\mathbb{S}^{N-1}|]^{-1}$. Set
\[G = G_\Omega,\ H = H_\Omega,\ G_\ep = G_{\Omega_\ep},\ H_\ep = H_{\Omega_\ep},\ G_{\ep,1} = G_{\R^N \setminus \overline{B(0,\ep)}},\ H_{\ep,1} = H_{\R^N\setminus \overline{B(0,\ep)}}.\]
Then $H_{\ep,1}$ is simply written as
\begin{equation}\label{H ep1}
H_{\ep,1}(x,y)= \ga_N\ep^{N-2}\left||y|\(x-\frac{\ep^2y}{|y|^2}\)\right|^{2-N} \quad \textup{for } x, y \in \R^N\setminus \overline{B(0,\ep)}.
\end{equation}
As the next result shows, one can decompose $H_\ep$ into two parts $H$ and $H_{\ep,1}$ modulo a small remainder term.
\begin{lemma}\label{reg est for puntured}
Let $N \ge 3$ and $\ep > 0$ small. For any $x,y \in \Omega_\ep$, it holds that
$$
H_\ep(x,y)=H(x,y) + H_{\ep,1}(x,y) + O(1)\ep^{N-2}\(|x|^{2-N}+|y|^{2-N}\).
$$
\end{lemma}
\begin{proof}
Fixing $y \in \Omega_\ep$, let $\mcr(x,y)=H_\ep(x,y)-H_{\ep,1}(x,y)-H(x,y)$, which solves
$$
\begin{cases}
-\Delta_x \mcr(x,y)=0 &\textup{for } x \in \Omega_\ep,\\
\mcr(x,y)= -H_{\ep,1}(x,y) &\textup{for } x \in \pa \Omega,\\
\mcr(x,y)= -H(x,y) &\textup{for } x \in \pa B(0,\ep).
\end{cases}
$$
Since
$$
\ep^{N-2}\left||y|\(x-\frac{\ep^2y}{|y|^2}\)\right|^{2-N} \le C \frac{\ep^{N-2}}{|y|^{N-2}} \textup{ for } x \in \pa \Omega \textup{ and } y \in \Omega_\ep
$$
and
$$
H(x,y)\le C\frac{\ep^{N-2}}{|x|^{N-2}} \textup{ for } x \in \pa B(0,\ep) \textup{ and } y \in \Omega_\ep,
$$
the maximum principle yields
\[\sup_{x \in \Omega_\ep}|\mcr(x,y)| \le \sup_{x \in \pa \Omega_\ep}|\mcr(x,y)|\le C\ep^{N-2}\(|x|^{2-N}+|y|^{2-N}\). \qedhere\]
\end{proof}

For later use, we introduce two functions: Assume that $N \ge 3$ and $p \in (\frac{2}{N-2},\frac{N}{N-2})$. Given $y \in \Omega$, let $\wtg_y:\Omega \to \R$ be a function such that
$$
\begin{cases}
-\Delta \wtg_y = G^p(\cdot,y) &\textup{in } \Omega,\\
\wtg_y = 0 &\textup{on } \pa \Omega
\end{cases}
$$
and $\wth_y:\Omega \to \R$ its regular part given by
\begin{equation}\label{wth}
\wth_y(x) = \frac{\tilde{\ga}_{N,p}}{|x-y|^{(N-2)p-2}}-\wtg_y(x) \quad \textup{for } x \in \Omega
\end{equation}
where
\begin{equation}\label{tilde gamma}
\tilde{\ga}_{N,p} := \frac{\ga_N^p}{((N-2)p-2)(N-(N-2)p)} > 0.
\end{equation}
The maximum principle implies that $\wth_y(x) > 0$ for all $x,y \in \Omega$. Thanks to the assumption that $p <\frac{N-1}{N-2}$, we have the following regularity result.
\begin{lemma}[Lemmas 2.1 and 2.11 in \cite{KP}]\label{wth reg}
Assume that $p \in (\frac{2}{N-2},\frac{N-1}{N-2})$ and let $\wth(x,y) = \wth_y(x)$ for $x,y \in \Omega$. Then
\begin{itemize}
\item[-] there exists $\sigma \in (0,1)$ such that $\|\wth_y\|_{C^{1,\sigma}(\overline{\Omega})}$ is uniformly bounded for $y$ in a compact subset of $\Omega$.
    In particular, for each $y \in \Omega$, the map $x \in \overline{\Omega} \mapsto (\nabla_x \wth)(x,y)$ is continuous;
\item[-] for each $x \in \Omega$, the map $y \in \Omega \mapsto (\nabla_y \wth)(x,y)$ is continuous.
\end{itemize}
\end{lemma}

\subsection{Limit system}
Let $N \ge 3$ and $(U,V)$ be the unique positive ground state solution to
\begin{equation}\label{limit system}
\begin{cases}
-\Delta u=|v|^{p-1}v \quad \textup{in } \R^N,\\
-\Delta v=|u|^{q-1}u \quad \textup{in } \R^N,\\
\displaystyle (u,v)\in \dot{W}^{2,\frac{p+1}{p}}(\R^N)\times \dot{W}^{2,\frac{q+1}{q}}(\R^N)
\end{cases}
\end{equation}
such that $u(0)=\max_{x \in \R^N} u(x)=1$. According to \cite{ALT}, it is radially symmetric and decreasing in the radial variable.

\begin{lemma}[Theorem 2 in \cite{HV}]
If $p \in (\frac{2}{N-2},\frac{N}{N-2})$, then
$$
\lim_{|x| \to \infty} |x|^{(N-2)p-2}U(x) = a_{N,p} \quad \textup{and} \quad \lim_{|x| \to \infty} |x|^{N-2}V(x) = b_{N,p}
$$
where $a_{N,p}$ and $b_{N,p}$ are positive constants depending only on $N$ and $p$ with
\begin{equation}\label{bnp anp}
b_{N,p}^p = a_{N,p}\((N-2)p-2\)\(N-(N-2)p\).
\end{equation}
\end{lemma}

\begin{lemma}[Corollaries 2.6 and 2.7 in \cite{KP}, Lemma 2.2 in \cite{KM}]\label{limit decay}
Suppose that $N \ge 3$ and $p \in (\frac{2}{N-2},\frac{N}{N-2})$. Then there exist some $\kappa_0, \kappa_1 \in (0,1)$ such that
\begin{equation}\label{decay for u}
\left|U(x)-\frac{a_{N,p}}{|x|^{(N-2)p-2}}\right| = O\(\frac{1}{|x|^{(N-2)p-1}}\),
\end{equation}
\begin{equation}\label{decay for v}
\left|V(x)-\frac{b_{N,p}}{|x|^{N-2}}\right| = O\(\frac{1}{|x|^{N-1}}\),
\end{equation}
\begin{equation}\label{decay for grad u}
\left|\nabla U(x)+((N-2)p-2)a_{N,p}\frac{x}{|x|^{(N-2)p}}\right|= O\(\frac{1}{|x|^{(N-2)p-1 +\kappa_0}}\),
\end{equation}
\begin{equation}\label{decay for grad v}
\left|\nabla V(x)+(N-2)b_{N,p}\frac{x}{|x|^N}\right| = O\(\frac{1}{|x|^{N-1+\kappa_1}}\)
\end{equation}
for $|x| \ge 1$.
\end{lemma}

Given $(\mu,\xi) \in (0,\infty) \times \R^N$, we define
\begin{equation}\label{UVmx}
(U_\mx(x),V_\mx(x)) = \(\mu^{-\frac{N}{q+1}}U(\mu^{-1}(x-\xi)),\mu^{-\frac{N}{p+1}}V(\mu^{-1}(x-\xi))\) \quad \textup{for } x \in \R^N
\end{equation}
so that $(U,V) = (U_{1,0},V_{1,0})$. According to \cite{Wa, HV}, they constitute the entire set of solutions to \eqref{limit system}. They are also non-degenerate as the next proposition states.

\begin{proposition}[Theorem 1 in \cite{FKP}]\label{nondeg}
We define
\begin{equation}\label{PsPhmxl10}
\begin{aligned}
\(\Psi_{1,0}^0(x),\Phi_{1,0}^0(x)\) &= \(x \cdot \nabla U(x) + \frac{NU(x)}{q+1},\, x \cdot \nabla V(x) + \frac{NV(x)}{p+1}\),\\
\big(\Psi_{1,0}^l(x),\Phi_{1,0}^l(x)\big) &= \(\pa_{x_l}U(x), \pa_{x_l}V(x)\)
\end{aligned}
\end{equation}
for $x \in \R^N$ and $l = 1,\ldots,N$. Also, given $\mu > 0$ and $\xi \in \R^N$, we set
\begin{equation}\label{PsPhmxl}
\big(\Psi_\mx^l(x), \Phi_\mx^l(x)\big) = \(\mu^{-\frac{N}{q+1}-1} \Psi_{1,0}^l (\mu^{-1}(x-\xi)), \mu^{-\frac{N}{p+1}-1} \Phi_{1,0}^l (\mu^{-1}(x-\xi))\)
\end{equation}
for $x \in \R^N$ and $l = 0,1,\ldots,N$. Then the space of solutions to the linear system
\[\begin{cases}
-\Delta \Psi = pV_\mx^{p-1} \Phi \quad \textup{in } \R^N,\\
-\Delta \Phi = qU_\mx^{q-1} \Psi \quad \textup{in } \R^N,\\
(\Psi,\Phi) \in \dot{W}^{2,\frac{p+1}{p}}(\R^N) \times \dot{W}^{2,\frac{q+1}{q}}(\R^N)
\end{cases}\]
is spanned by
\[\left\{\(\Psi_\mx^0, \Phi_\mx^0\), \(\Psi_\mx^1,\Phi_\mx^1\), \ldots, \(\Psi_\mx^N,\Phi_\mx^N\)\right\}.\]
\end{proposition}

\section{Approximation of the solution}\label{sec app}
\subsection{Admissible set of parameters}
Given a small fixed number $\delta \in (0,1)$ to be determined in Section \ref{sec comp}, we define the admissible set of parameters
\begin{equation}\label{Lambda delta}
\Lambda_\delta = \left\{(d,\tau): d\in [\delta,\delta^{-1}],\ \tau = (\tau_1,\ldots,\tau_N) \in \overline{B(0,\delta^{-1})}\right\}.
\end{equation}
Also, letting
\begin{equation}\label{mu ep}
\mu_\ep = \ep^\alpha \quad \textup{where } \alpha = \frac{N-2}{(N-2)p+N-4} \in (0,1) \quad \textup{for } N \ge 4 \textup{ and } p > 1,
\end{equation}
we write
\begin{equation}\label{trans and scaling}
\mu = \mu_\ep d \quad \textup{and} \quad \xi = \mu\tau = \mu_\ep d\tau.
\end{equation}
The following relations will be often useful:
\begin{equation}\label{simple rel}
\frac{N(p+1)}{q+1}=(N-2)p-2 \quad \text{and} \quad \(\frac{\ep}{\mu_\ep}\)^{N-2} = \mu_\ep^{(N-2)p-2}.
\end{equation}

\subsection{The first approximation of the solution}
We analyze the linear projection of $(U_\mx,V_\mx)$ in \eqref{UVmx} into the space $W_0^{1,p^\ast}(\Omega_\ep) \times W_0^{1,q^\ast}(\Omega_\ep)$; see \eqref{pq star} for the definition of the pair $(p^*,q^*)$.
It will serve as the first approximation of the positive solution to \eqref{main system}. More precisely, we consider the system
\begin{equation}\label{PUPV eq}
\begin{cases}
-\Delta PU_\mx=V_\mx^p &\textup{in } \Omega_\ep,\\
-\Delta PV_\mx=U_\mx^q &\textup{in } \Omega_\ep,\\
PU_\mx=PV_\mx=0 &\textup{on } \pa \Omega_\ep.
\end{cases}
\end{equation}
\begin{lemma}\label{first app}
Assume $N \ge 4$, $p \in (1,\frac{N}{N-2})$, and $\ep > 0$ small. Let $\whh:\Omega \times \Omega\to \R$ be a smooth function such that
\begin{equation}\label{whh}
\begin{cases}
-\Delta_x \whh(x,y)=0 &\textup{for } x \in \Omega,\\
\displaystyle \whh(x,y)=\frac{1}{|x-y|^{(N-2)p-2}} &\textup{for } x \in \pa \Omega.
\end{cases}
\end{equation}
Then, for $x \in \Omega_\ep$,
\begin{align}
PU_\mx(x) &= U_\mx(x) - a_{N,p}\, \mu^{\frac{Np}{q+1}}\whh(x,\xi) - \mu^{-\frac{N}{q+1}}U(\tau)\frac{\ep^{N-2}}{|x|^{N-2}} + R_1(x), \label{approx for pu}\\
PV_\mx(x) &= V_\mx(x) - \frac{b_{N,p}}{\ga_N}\mu^{\frac{N}{q+1}}H(x,\xi) - \mu^{-\frac{N}{p+1}}V(\tau)\frac{\ep^{N-2}}{|x|^{N-2}} + R_2(x). \label{approx for pv}
\end{align}
Here, $R_1$ and $R_2$ are the remainder terms such that
\begin{align*}
|R_1(x)| &= O(1)\left[\mu^{\frac{Np}{q+1}} \mu + \ep^{N-2}\mu^{-\frac{N}{q+1}}|x|^{2-N}\(\mu^{(N-2)p-2} + \frac{\ep}{\mu}\)\right],\\
|R_2(x)| &= O(1)\left[\mu^{\frac{N}{q+1}} \left\{\mu + \(\frac{\ep}{\mu}\)^{N-2}\right\} + \ep^{N-2}\mu^{-\frac{N}{p+1}} |x|^{2-N} \(\mu^{N-2}+\frac{\ep}{\mu}\)\right]
\end{align*}
where $C > 0$ in \eqref{O(1)} is chosen uniformly for $(d,\tau)\in \Lambda_\delta$.
\end{lemma}
\begin{proof}
We will treat \eqref{approx for pu} only, because the proof of \eqref{approx for pv} is similar. First, we see that $R_1$ is a harmonic function in $\Omega_\ep$. We also have
\begin{align*}
R_1(x)&=\mu^{-\frac{N}{q+1}} \left[U(\tau)-U\(\frac{x-\xi}{\mu}\)\right] + a_{N,p}\, \mu^{\frac{Np}{q+1}}\whh(x,\xi) \\
&=\mu^{-\frac{N}{q+1}} \left[U(\tau)-U\(\tau-\frac{\ep}{\mu}\frac{x}{|x|}\)\right] + O(1)\mu^{\frac{Np}{q+1}}
= O(1)\mu^{-\frac{N}{q+1}}\(\mu^{(N-2)p-2}+\frac{\ep}{\mu}\)
\end{align*}
for $x \in \pa B(0,\ep)$, and by \eqref{decay for u} and \eqref{simple rel},
\begin{align*}
R_1(x)&=-\mu^{-\frac{N}{q+1}}U (\mu^{-1}(x-\xi))+\frac{a_{N,p}\, \mu^{\frac{Np}{q+1}}}{|x-\xi|^{(N-2)p-2}} - \mu^{-\frac{N}{q+1}}U(\tau)\frac{\ep^{N-2}}{|x|^{N-2}}\\
&=\frac{\mu^{\frac{Np}{q+1}}}{|x-\xi|^{(N-2)p-2}} \left[-U\(\mu^{-1}(x-\xi)\) \left|\mu^{-1}(x-\xi)\right|^{(N-2)p-2}+a_{N,p}\right] + O(1)\mu^{-\frac{N}{q+1}}\ep^{N-2}\\
&= O(1)\mu^{1+\frac{Np}{q+1}} + O(1)\mu^{-\frac{N}{q+1}}\ep^{N-2} = O(1) \mu^{\frac{Np}{q+1}} \left[\mu+\(\frac{\ep}{\mu}\)^{N-2}\mu^{N-(N-2)p}\right] = O(1) \mu^{\frac{Np}{q+1}} \mu
\end{align*}
for $x \in \pa \Omega$. Thus the maximum principle gives the upper bound for $|R_1|$.
%
\end{proof}

\begin{cor}\label{rem C1}
Under the hypotheses of Lemma \ref{first app}, we have
\begin{align*}
|\nabla_{(d,\tau)} R_1(x)|
& = O(1) \left[\mu^{\frac{Np}{q+1}} \mu^{\kappa_0} + \ep^{N-2}\mu^{-\frac{N}{q+1}}|x|^{2-N}\(\mu^{(N-2)p-2} + \frac{\ep}{\mu}\)\right], \\
|\nabla_{(d,\tau)} R_2(x)|
& = O(1) \left[\mu^{\frac{N}{q+1}} \mu^{\kappa_1} + \ep^{N-2}\mu^{-\frac{N}{p+1}} |x|^{2-N} \(\mu^{N-2} + \frac{\ep}{\mu}\)\right],
\end{align*}
for all $x \in \Omega_\ep$, where $C > 0$ in \eqref{O(1)} is chosen uniformly for $(d,\tau) \in \Lambda_\delta$, and $\kappa_0, \kappa_1 \in (0,1)$ are the numbers in \eqref{limit decay}.
\end{cor}
\begin{proof}
Using \eqref{decay for grad u} and \eqref{decay for grad v}, one can argue as in the proof of previous lemma.
\end{proof}

\subsection{The second approximation of the solution}
It turns out that the $u$-component $PU_\mx$ of the first approximation $(PU_\mx, PV_\mx)$ makes a huge error and we must refine it to build an actual solution to \eqref{main system}.
Our idea is, motivated by Subsection 2.3 in \cite{KP}, to employ the function $\mcp U_\mx$ solving
\begin{equation}\label{mcp U eq}
\begin{cases}
-\Delta \mcp U_\mx = (PV_\mx)^p &\textup{in } \Omega_\ep,\\
\mcp U_\mx =0 &\textup{on } \pa \Omega_\ep
\end{cases}
\end{equation}
as the second (i.e., refined) approximation for the $u$-component of the solution.
\begin{proposition}\label{second approx}
Assume $N \ge 4$, $p \in (1,\frac{N}{N-2})$, and $\ep > 0$ small. Given an auxiliary number $\kappa \in (0,1)$ small enough (determined by $N$ and $p$)\,\footnote{Working with $\kappa = 0$
gives the remainder term in the energy expansion in \eqref{ene exp} and \eqref{ene exp C1} too large.} and the parameters $(d,\tau) \in \Lambda_\delta$, we set
\begin{equation}\label{A mu}
\mca_{\ep,(d,\tau)}(x) = \int_{B(0,\mu^{\kappa-1})\setminus B(-\tau,\frac{\ep}{\mu})} \frac{G_\ep(x,\mu y+\xi)V^{p-1}(y)}{|y+\tau|^{N-2}}\, \dy \quad \textup{for } x \in \Omega_\ep
\end{equation}
where $(\mu,\xi) \in (0,\infty) \times \R^N$ is defined in \eqref{trans and scaling}. Then it holds that
\begin{equation}\label{mcp U}
\begin{aligned}
\mcp U_\mx(x) = U_\mx(x)
&- \(\frac{b_{N,p}}{\ga_N}\)^p\mu^{\frac{Np}{q+1}}\wth_\xi(x) \\
&- \ep^{N-2}\mu^{-\frac{N}{q+1}} \left[U(\tau) {|x|^{2-N}} +p V(\tau) \mca_{\ep,(d,\tau)}(x)\right] + R(x) \quad \textup{for } x \in \Omega_\ep.
\end{aligned}
\end{equation}
Here, $\wth_\xi$ is the function in \eqref{wth} with $y = \xi$ and
\begin{equation}\label{R}
R(x) := O(1) \ep^{(N-2)p}\mu^{-\frac{Np}{p+1}}|x|^{2-(N-2)p} + o(1)\left[\mu^{\frac{N p}{q+1}} + \ep^{N-2}\mu^{-\frac{N}{q+1}} \({|x|^{2-N}}+\mca_{\ep,(d,\tau)}(x)\)\right]
\end{equation}
where $C,\, C_\ep > 0$ in \eqref{O(1)} and \eqref{o(1)} are chosen uniformly for $(d,\tau)\in \Lambda_\delta$. Furthermore,
\begin{equation}\label{mca est}
\mca_{\ep,(d,\tau)}(x) = O(1)\mu^{(N-2)p-N}|x|^{2-(N-2)p} \quad \textup{for } x \in \Omega_\ep
\end{equation}
where $C > 0$ in \eqref{O(1)} is chosen uniformly for $(d,\tau)\in \Lambda_\delta$.
\end{proposition}
\noindent In contrast to the previous subsection, it is difficult to estimate $\mcp U_\mx$ by exploiting the maximum principle. Instead, we conduct the potential analysis. More precisely, we estimate the right-hand side of the formula
\begin{equation}\label{gr rep}
\begin{aligned}
\(\mcp U_\mx-PU_\mx\)(x) &=\int_{\Omega_\ep \cap B(\xi,\mu^\kappa)} G_\ep(x,y)\(PV_\mx^p-V_\mx^p\)(y)\dy \\
&\ +\int_{\Omega_\ep\setminus B(\xi,\mu^\kappa)} G_\ep(x,y) \(PV_\mx^p-V_\mx^p\)(y)\dy,
\end{aligned}
\end{equation}
which holds for all $x \in \Omega_\ep$ owing to Green's representation formula.

In the setting of \cite{KP}, the first integral in the right-hand side of \eqref{gr rep} is ignorable,
and the leading-order term of the second integral is a multiple of $\wth_\xi$.
In our case, the leading-order term of the first integral is a multiple of $\mca_{\ep,(d,\tau)}$ (so the first integral \textbf{does} contribute to the expansion of $\mcp U_\mx-PU_\mx$), while that of the second integral is a linear combination of $\wth_\xi(x)$ and $|x|^{2-N}$.

On the other hand, by analyzing \eqref{A mu}, one sees that $\mca_{\ep,(d,\tau)}(x) \ll |x|^{2-N}$ for $\ep < |x| \ll \mu$ (refer to \eqref{mca est}) and $\mca_{\ep,(d,\tau)}(x) \gg |x|^{2-N}$ for $x \in \Omega_\ep$ away from 0, so $\mca_{\ep,(d,\tau)}(x)$ and $|x|^{2-N}$ are uncomparable in the pointwise sense.
Nonetheless, they yield the same-order terms in the expansion \eqref{ene exp} and \eqref{ene exp C1} of the reduced energy $J_\ep$ in \eqref{red ene}.
Remarkably, the seemingly `ugly' function $\mca_{\ep,(d,\tau)}$ produces a very neat term \eqref{second term est4} during the process of energy expansion.

\begin{proof}[Proof of Proposition \ref{second approx}]
Let us estimate the first integral on the right-hand side of \eqref{gr rep}. Using \eqref{approx for pv}, \eqref{A mu} and the first inequality in \eqref{ele ineq}, we evaluate
\begin{equation}\label{inner est}
\begin{aligned}
&\ \int_{\Omega_\ep \cap B(\xi,\mu^\kappa)} G_\ep(x,y)\(PV_\mx^p-V_\mx^p\)(y)\dy \\
&= \int_{\Omega_\ep \cap B(\xi,\mu^\kappa)} G_\ep(x,y) \left[p V_\mx^{p-1}(y) \left\{-\frac{b_{N,p}}{\ga_N}\mu^{\frac{N}{q+1}}H(y,\xi) - \mu^{-\frac{N}{p+1}}V(\tau)\frac{\ep^{N-2}}{|y|^{N-2}} + R_2(y)\right\} \right. \\
&\hspace{100pt} \left. + O(1)\(\mu^{\frac{Np}{q+1}} + \mu^{-\frac{pN}{p+1}} \frac{\ep^{(N-2)p}}{|y|^{(N-2)p}}\)\right] \dy\\
&=-\left(1+o(1)\right) \ep^{N-2}\mu^{-\frac{N}{p+1}} pV(\tau) \int_{\Omega_\ep \cap B(\xi,\mu^\kappa)}\frac{G_\ep(x,y) V_\mx^{p-1}(y)}{|y|^{N-2}}\, \dy\\
&\ + O(1) \int_{\Omega_\ep \cap B(\xi,\mu^\kappa)}\frac{1}{|x-y|^{N-2}} \(V_\mx^{p-1}(y) \mu^{\frac{N}{q+1}} + \mu^{\frac{Np}{q+1}} + \mu^{-\frac{pN}{p+1}} \frac{\ep^{(N-2)p}}{|y|^{(N-2)p}}\)\dy\\
&=-(1+o(1)) \ep^{N-2}\mu^{-\frac{N}{q+1}} pV(\tau) \mca_{\ep,(d,\tau)}(x) + O(1) \ep^{(N-2)p}\mu^{-\frac{pN}{p+1}} |x|^{2-(N-2)p} + o(1)\mu^{\frac{Np}{q+1}}
\end{aligned}
\end{equation}
for $x \in \Omega_\ep$. Here, the third equality is valid because of \eqref{interior es},
\begin{align*}
\mu^{-\frac{N}{p+1}} \int_{\Omega_\ep \cap B(\xi,\mu^\kappa)}\frac{G_\ep(x,y)V_\mx^{p-1}(y)}{|y|^{N-2}}\, \dy
&= \mu^{-\frac{N}{q+1}} \int_{B(0,\mu^{\kappa-1})\setminus B(-\tau,\frac{\ep}{\mu})}\frac{G_\ep(x,\mu y+\xi)V^{p-1}(y)}{ |y+\tau|^{N-2}}\, \dy \\
&= \mu^{-\frac{N}{q+1}} \mca_{\ep,(d,\tau)}(x)
\end{align*}
and
\[\int_{\Omega_\ep \cap B(\xi,\mu^\kappa)}\frac{\dy}{|x-y|^{N-2}|y|^{(N-2)p}} \le \int_{\R^N} \frac{\dy}{|x-y|^{N-2}|y|^{(N-2)p}} = O(1)|x|^{2-(N-2)p}.\]

\medskip
We next estimate the second integral on the right-hand side of \eqref{gr rep}. For $y \in \Omega_\ep\setminus B(\xi,\mu^\kappa)$, we have
\begin{equation}\label{v exp out}
\begin{aligned}
V_\mx(y) &= \mu^{-\frac{N}{p+1}}V(\mu^{-1}(y-\xi))\\
&= \mu^{-\frac{N}{p+1}} \left[b_{N,p}|\mu^{-1}(y-\xi)|^{2-N}+O(1)|\mu^{-1}(y-\xi)|^{1-N}\right] &(\textup{by } \eqref{decay for v})\\
&= \mu^{\frac{N}{q+1}} \left[b_{N,p}|y-\xi|^{2-N}+O(1)\mu|y-\xi|^{1-N}\right] &(\textup{by } \eqref{critical}),
\end{aligned}
\end{equation}
and so
\begin{equation}\label{pv outer}
\begin{aligned}
PV_\mx(y) &= V_\mx(y) - \frac{b_{N,p}}{\ga_N}\mu^{\frac{N}{q+1}}H(y,\xi) - \mu^{-\frac{N}{p+1}}V(\tau) \frac{\ep^{N-2}}{|y|^{N-2}} + R_2(y) &(\textup{by } \eqref{approx for pv})\\
&=\frac{b_{N,p}}{\ga_N}\mu^{\frac{N}{q+1}} \left[\ga_N|y-\xi|^{2-N}-H(y,\xi)\right] \\
&\ -\mu^{-\frac{N}{p+1}}V(\tau)\frac{\ep^{N-2}}{|y|^{N-2}} + O(1)\mu^{\frac{N}{q+1}+1}|y-\xi|^{1-N} + R_2(y) &(\textup{by }\eqref{v exp out})\\
&=\frac{b_{N,p}}{\ga_N}\mu^{\frac{N}{q+1}}G(y,\xi) - \mu^{-\frac{N}{p+1}}V(\tau)\frac{\ep^{N-2}}{|y|^{N-2}} + \wtr_2(y)
\end{aligned}
\end{equation}
where
\begin{equation}\label{wtr 2}
\begin{aligned}
\wtr_2(y) &:= O(1)\mu^{\frac{N}{q+1}+1}|y-\xi|^{1-N}+R_2(y)\\
&=O(1)\left[\mu^{\frac{N}{q+1}} \(\mu^{1+\kappa(1-N)}+\(\frac{\ep}{\mu}\)^{N-2}\) + \ep^{N-2}\mu^{-\frac{N}{p+1}} |y|^{2-N} \(\mu^{N-2}+\frac{\ep}{\mu}\)\right].
\end{aligned}
\end{equation}
Thus, for any $x \in \Omega_\ep$,
\begin{equation}\label{u rep}
\begin{aligned}
&\ \int_{\Omega_\ep\setminus B(\xi,\mu^\kappa)} G_\ep(x,y) V_\mx^p(y)\dy \\
&= b_{N,p}^p\mu^{\frac{Np}{q+1}} \int_{\Omega_\ep\setminus B(\xi,\mu^\kappa)} G_\ep(x,y) |y-\xi|^{(2-N)p} \(1+O(1)\mu^{1-\kappa}\)^p\dy &(\textup{by } \eqref{v exp out}) \\
&= b_{N,p}^p\mu^{\frac{Np}{q+1}} \int_{\Omega_\ep\setminus B(\xi,\mu^\kappa)} G_\ep(x,y) |y-\xi|^{(2-N)p}\, \dy \\
&\ + O(1)\mu^{\frac{Np}{q+1}}\mu^{1-\kappa}\int_{\Omega \setminus B(0,\frac{1}{2}\mu^\kappa)}|x-y|^{2-N} |y|^{(2-N)p}\, \dy \\
&= b_{N,p}^p\mu^{\frac{Np}{q+1}} \int_{\Omega_\ep\setminus B(\xi,\mu^\kappa)} G_\ep(x,y) |y-\xi|^{(2-N)p}\, \dy + O(1)\mu^{\frac{Np}{q+1}}\mu^{(1-\kappa)+ \kappa(2-(N-2)p)} &(\textup{by } \eqref{basic est}),
\end{aligned}
\end{equation}
and by \eqref{pv outer} and the second inequality in \eqref{ele ineq},
\begin{equation}\label{u rep 1}
\begin{aligned}
&\ \int_{\Omega_\ep\setminus B(\xi,\mu^\kappa)} G_\ep(x,y) PV_\mx^p(y)\dy \\
&= \int_{\Omega_\ep\setminus B(\xi,\mu^\kappa)} G_\ep(x,y) \(\frac{b_{N,p}}{\ga_N}\mu^{\frac{N}{q+1}}G(y,\xi)\)^p \dy \\
&\ -\int_{\Omega_\ep\setminus B(\xi,\mu^\kappa)} G_\ep(x,y) \left|\mu^{-\frac{N}{p+1}}V(\tau) \frac{\ep^{N-2}}{|y|^{N-2}}-\wtr_2(y)\right|^{p-1} \(\mu^{-\frac{N}{p+1}}V(\tau) \frac{\ep^{N-2}}{|y|^{N-2}}-\wtr_2(y)\) \dy \\
&\ + O(1)\int_{\Omega_\ep\setminus B(\xi,\mu^\kappa)} G_\ep(x,y) \left[\mu^{\frac{N}{q+1}}G(y,\xi) \right]^{p-1} \left|\mu^{-\frac{N}{p+1}} \frac{\ep^{N-2}}{|y|^{N-2}} -\wtr_2(y) \right|\dy \\
&= \(\frac{b_{N,p}}{\ga_N}\)^p\mu^{\frac{Np}{q+1}}\int_{\Omega_\ep\setminus B(\xi,\mu^\kappa)} G_\ep(x,y)G^p(y,\xi)\dy
- \ep^{(N-2)p}\mu^{-\frac{Np}{p+1}}V^p(\tau) \int_{\Omega_\ep\setminus B(\xi,\mu^\kappa)} \frac{G_\ep(x,y)}{|y|^{(N-2)p}}\, \dy\\
&\ + O(1)R_3(x)
\end{aligned}
\end{equation}
where
\begin{equation}\label{R3}
\begin{aligned}
R_3(x) := \int_{\Omega_\ep\setminus B(\xi,\mu^\kappa)} G_\ep(x,y) &\left[\left[\mu^{\frac{N}{q+1}}G(y,\xi) \right]^{p-1}\(\mu^{-\frac{N}{p+1}}\frac{\ep^{N-2}}{|y|^{N-2}} +\big|\wtr_2(y)\big|\) \right] \\
&\left. + \(\mu^{-\frac{N}{p+1}} \frac{\ep^{N-2}}{|y|^{N-2}}\)^{p-1} \big|\wtr_2(y)\big|+\big|\wtr_2(y)\big|^p \right]\dy.
\end{aligned}
\end{equation}
Owing to Lemma \ref{high est} and \eqref{basic est whole}, it holds that
\begin{equation}\label{R3 est}
\begin{cases}
\displaystyle R_3(x)= o(1)\mu^{\frac{Np}{q+1}},\\
\displaystyle \int_{\Omega_\ep\setminus B(\xi,\mu^\kappa)} \frac{G_\ep(x,y)}{|y|^{(N-2)p}}\, \dy = O(1) \int_{\R^N} \frac{\dy}{|x-y|^{N-2}|y|^{(N-2)p}} = O(1)|x|^{2-(N-2)p}
\end{cases}
\end{equation}
for $\kappa \in (0,1)$ small. Hence, we see from \eqref{u rep}, \eqref{u rep 1} and \eqref{R3 est} that
$$
\int_{\Omega_\ep\setminus B(\xi,\mu^\kappa)} G_\ep(x,y) V_\mx^p(y)\dy = b_{N,p}^p\mu^{\frac{Np}{q+1}} \int_{\Omega_\ep\setminus B(\xi,\mu^\kappa)} G_\ep(x,y) |y-\xi|^{(2-N)p}\, \dy +o(1)\mu^{\frac{Np}{q+1}}
$$
and
\begin{align*}
\int_{\Omega_\ep\setminus B(\xi,\mu^\kappa)} G_\ep(x,y) PV_\mx^p(y)\dy
&=\(\frac{b_{N,p}}{\ga_N}\)^p\mu^{\frac{Np}{q+1}}\int_{\Omega_\ep\setminus B(\xi,\mu^\kappa)} G_\ep (x,y)G^p(y,\xi)\dy\\
&\ + O(1) \ep^{(N-2)p}\mu^{-\frac{Np}{p+1}}|x|^{2-(N-2)p} + o(1)\mu^{\frac{Np}{q+1}}
\end{align*}
for $\kappa \in (0,1)$ small. Combining these yields
\begin{equation}\label{ot est}
\begin{aligned}
&\ \int_{\Omega\setminus B(\xi,\mu^\kappa)} G_\ep(x,y)\(PV_\mx^p-V_\mx^p\)(y)\dy\\
&= \(\frac{b_{N,p}}{\ga_N}\)^p\mu_\ep^{\frac{Np}{q+1}}d^{\frac{Np}{q+1}} \int_{\Omega\setminus B(\xi,\mu^\kappa)} G_\ep(x,y)\(G^p(y,\xi)-\ga_N^p|y-\xi|^{(2-N)p}\)\dy\\
&\ +O(1) \ep^{(N-2)p}\mu^{-\frac{Np}{p+1}}|x|^{2-(N-2)p} + o(1)\mu^{\frac{Np}{q+1}}.
\end{aligned}
\end{equation}
We write
\begin{equation}\label{g ep to g}
\begin{aligned}
&\ \int_{\Omega\setminus B(\xi,\mu^\kappa)} G_\ep(x,y)\(G^p(y,\xi)-\ga_N^p|y-\xi|^{(2-N)p}\)\dy \\
&= \int_{\Omega} G(x,y) \(G^p(y,\xi)-\ga_N^p|y-\xi|^{(2-N)p}\)\dy - \int_{B(\xi,\mu^\kappa)} G(x,y) \(G^p(y,\xi)-\ga_N^p|y-\xi|^{(2-N)p}\)\dy \\
&\ - \int_{\Omega\setminus B(\xi,\mu^\kappa)}(H_{\ep,1}(x,y)+\mcr(x,y)) \(G^p(y,\xi)-\ga_N^p|y-\xi|^{(2-N)p}\)\dy
\end{aligned}
\end{equation}
where $\mcr(x,y)=H_\ep(x,y)-H_{\ep,1}(x,y)-H(x,y)$. We observe from the first inequality in \eqref{ele ineq} that
\begin{align*}
&\ \int_{B(\xi,\mu^\kappa)} G(x,y) \(G^p(y,\xi)-\ga_N^p|y-\xi|^{(2-N)p}\) \dy\\
&=\ga_N^p\int_{B(\xi,\mu^\kappa)} G(x,y)|y-\xi|^{(2-N)p}\left[\(1-\ga_N^{-1}|y-\xi|^{N-2}H(y,\xi)\)^p-1\right] \dy\\
&=O(1)\int_{B(\xi,\mu^\kappa)} \frac{\dy}{|x-y|^{N-2}|y-\xi|^{(N-2)(p-1)}} = O(1)\mu^{\kappa(N-(N-2)p)} &( \textup{by } \eqref{interior es})
\end{align*}
and
\begin{equation}\label{ot est2}
\begin{aligned}
&\ \int_{\Omega\setminus B(\xi,\mu^\kappa)} (H_{\ep,1}(x,y)+\mcr(x,y))\(G^p(y,\xi)-\ga_N^p|y-\xi|^{(2-N)p}\)\dy\\
&=O(1)\int_{\Omega\setminus B(\xi,\mu^\kappa)} (H_{\ep,1}(x,y)+\mcr(x,y))(|y-\xi|^{(N-2)(1-p)}+1) \dy\\
&=O(1)\ep^{N-2}|x|^{2-N}\int_{\Omega\setminus B(0,\frac{1}{2}\mu^\kappa)} |y|^{-(N-2)p} \dy = O(1)\ep^{N-2}|x|^{2-N}
\end{aligned}
\end{equation}
where the second equation of \eqref{ot est2} follows from Lemma \ref{reg est for puntured} and
$$
\left||y|\(x-\frac{\ep^2y}{|y|^2}\)\right|^{2-N} = |x|^{2-N}|y|^{2-N}\left|\frac{x}{|x|}-\frac{\ep^2y}{|x||y|^2}\right|^{2-N} = O(1)(1+\ep \mu^{-\kappa})|x|^{2-N}|y|^{2-N}
$$
for $x \in \Omega_\ep$ and $y \in \Omega\setminus B(\xi,\mu^\kappa)$. Therefore, if we set $\varphi_\xi$ as the solution to
\begin{equation}\label{varphi}
\begin{cases}
-\Delta \varphi_\xi = G^p(\cdot,\xi)-\ga_N^p|\cdot-\xi|^{(2-N)p} &\textup{in } \Omega,\\
\varphi_\xi=0 &\textup{on } \pa \Omega,
\end{cases}
\end{equation}
then
\begin{multline}\label{g ep to varp}
\int_{\Omega\setminus B(\xi,\mu^\kappa)} G_\ep(x,y)\(G^p(y,\xi)-\ga_N^p|y-\xi|^{(2-N)p}\)\dy\\
=\varphi_\xi(x)+ O(1)\mu^{\kappa(N-(N-2)p)} +O(1)\ep^{N-2}|x|^{2-N}.
\end{multline}
Plugging this into \eqref{ot est}, we conclude
\begin{multline}\label{outer est}
\int_{\Omega\setminus B(\xi,\mu^\kappa)} G_\ep(x,y)\(PV_\mx^p-V_\mx^p\)(y)\dy =\(\frac{b_{N,p}}{\ga_N}\)^p\mu^{\frac{Np}{q+1}}\varphi_\xi(x) \\
+ O(1) \ep^{(N-2)p}\mu^{-\frac{Np}{p+1}}|x|^{2-(N-2)p} + o(1)\mu^{\frac{Np}{q+1}} + O(1)\ep^{N-2}|x|^{2-N}.
\end{multline}

As a result, we infer from \eqref{gr rep}, \eqref{approx for pu}, \eqref{inner est}, \eqref{A mu}, \eqref{outer est}, \eqref{bnp anp} and the identity
\[\wth_\xi(x) = \tilde{\ga}_{N,p}\whh(x,\xi)-\varphi_\xi(x) \quad \textup{(by \eqref{wth}, \eqref{tilde gamma}, \eqref{whh} and \eqref{varphi})}\]
that
\begin{align*}
\mcp U_\mx(x) &= U_\mx(x)-a_{N,p}\, \mu_\ep^{\frac{Np}{q+1}}d^{\frac{Np}{q+1}}\whh(x,\xi) + \(\frac{b_{N,p}}{\ga_N}\)^p\mu^{\frac{Np}{q+1}}\varphi_\xi(x) - \mu^{-\frac{N}{q+1}}U(\tau)\frac{\ep^{N-2}}{|x|^{N-2}} \\
&\ -(1+o(1))\ep^{N-2}\mu^{-\frac{N}{q+1}}pV(\tau) \mca_{\ep,(d,\tau)}(x) + O(1)\ep^{(N-2)p}\mu^{-\frac{Np}{p+1}} |x|^{2-(N-2)p}\\
&\ + o(1)\mu^{\frac{Np}{q+1}} + o(1)\ep^{N-2}\mu^{-\frac{N}{q+1}}|x|^{2-N}\\
&=U_\mx(x)- \(\frac{b_{N,p}}{\ga_N}\)^p\mu^{\frac{Np}{q+1}}\wth_\xi(x) -\mu^{-\frac{N}{q+1}}U(\tau)\frac{\ep^{N-2}}{|x|^{N-2}} \\
&\ - (1+o(1))\ep^{N-2}\mu^{-\frac{N}{q+1}}pV(\tau) \mca_{\ep,(d,\tau)}(x) + O(1)\ep^{(N-2)p}\mu^{-\frac{Np}{p+1}}|x|^{2-(N-2)p} + o(1)\mu^{\frac{Np}{q+1}}\\
&\ + o(1)\ep^{N-2}\mu^{-\frac{N}{q+1}}|x|^{2-N}.
\end{align*}
This concludes the proof of \eqref{mcp U}.

Finally, we note from \eqref{A mu} that
\begin{align*}
\mca_{\ep,(d,\tau)}(x) &= O(1)\mu^{2-N}\int_{\R^N} \frac{\dy}{|\mu^{-1}(x-\xi)-y|^{N-2}|y+\tau|^{(N-2)p}} \\
&= O(1)\mu^{2-N}|\mu^{-1}(x-\xi)+\tau|^{2-(N-2)p}=O(1)\mu^{(N-2)p-N}|x|^{2-(N-2)p},
\end{align*}
so \eqref{mca est} is true.
\end{proof}

The papers \cite{KP, KM} show that the single- and multi-bubble solutions to slightly subcritical Lane-Emden systems are characterized as local minima of the reduced energy.
Hence, to construct them, one only needs the $C^0$-estimate (corresponding to our \eqref{ene exp}) of the reduced energy.
In contrast, the solutions to \eqref{main system} found here will be characterized as saddle points of the reduced energy, which forces us to derive its $C^1$-estimate.
The following corollary, whose proof is not a mere adaptation of the proof of Proposition \ref{second approx}, roles as a key result in such estimate; refer to Proposition \ref{C1 est} and Lemma \ref{J0ep Jep}.
\begin{proposition}\label{second approx2}
Under the hypotheses of Proposition \ref{second approx}, we have
\begin{multline}\label{R C1}
\left|\nabla_{(d,\tau)} R(x)\right| = O(1) \ep^{(N-2)p}\mu^{-\frac{Np}{p+1}}|x|^{2-(N-2)p} \\
+ o(1)\left[\mu^{\frac{N p}{q+1}} + \ep^{N-2}\mu^{-\frac{N}{q+1}} \({|x|^{2-N}}+\mca_{\ep,(d,\tau)}(x)\)\right] \quad \textup{for } x \in \Omega_\ep
\end{multline}
where $C,\, C_\ep > 0$ in \eqref{O(1)} and \eqref{o(1)} are chosen uniformly for $(d,\tau) \in \Lambda_\delta$. Furthermore,
\begin{equation}\label{mca est C1}
\left|\nabla_{(d,\tau)} \mca_{\ep,(d,\tau)}(x)\right| = O(1)\mu^{(N-2)p-N}|x|^{2-(N-2)p} \quad \textup{for } x \in \Omega_\ep
\end{equation}
where $C > 0$ in \eqref{O(1)} is chosen uniformly for $(d,\tau)\in \Lambda_\delta$.
\end{proposition}
\begin{proof}
Here, we only consider the differentiation of the maps $R$ and $\mca_{\ep,(d,\tau)}$ with respect to the $d$-variable,
since the differentiation with respect to the $\tau$-variable can be handled in a similar way.

\medskip
Let us prove \eqref{R C1} (where $\nabla_{(d,\tau)}$ is replaced with $\nabla_d$).
Thanks to Corollary \ref{rem C1}, we only have to estimate the map $(d,\tau) \in \Lambda_\delta \mapsto \mcp U_\mx-PU_\mx$. We write
\begin{equation}\label{der est}
\begin{aligned}
&\ \nabla_d(\mcp U_\mx-PU_\mx)(x) \\
&= p\int_{\Omega_\ep \cap B(\xi,\mu^\kappa)} G_\ep(x,y) \left[\(\nabla_d V_\mx\) \(PV_\mx^{p-1}-V_\mx^{p-1}\) + PV_\mx^{p-1} \left\{\nabla_d\(PV_\mx-V_\mx\)\right\}\right](y)\dy \\
&\ +p\int_{\Omega_\ep \setminus B(\xi,\mu^\kappa)} G_\ep(x,y) \left[PV_\mx^{p-1} \(\nabla_d PV_\mx\) - V_\mx^{p-1} \(\nabla_d V_\mx\)\right](y)\dy
\end{aligned}
\end{equation}
for $x \in \Omega_\ep$.

Let us estimate the first term in the right-hand side of \eqref{der est}. Using the previous computations in \eqref{inner est},
the fact that $|\nabla_d V_\mx| \le CV_\mx$ in $\R^N$ for some $C > 0$ depending only on $N$ and $p$, and the last inequality in \eqref{ele ineq}, we deduce that
\begin{equation}\label{der est1}
\begin{aligned}
&\ \int_{\Omega_\ep \cap B(\xi,\mu^\kappa)} G_\ep(x,y) (\nabla_d V_\mx)(y) \(PV_\mx^{p-1}-V_\mx^{p-1}\)(y)\dy \\
&= (p-1)\int_{\Omega_\ep \cap B(\xi,\mu^\kappa)} G_\ep(x,y) \left[(\nabla_d V_\mx) V_\mx^{p-2} \(PV_\mx-V_\mx\) +O(1)|PV_\mx-V_\mx|^p\right](y)\dy \\
&= -(1+o(1)) \ep^{N-2}\mu^{-\frac{N}{p+1}} V(\tau) \int_{\Omega_\ep \cap B(\xi,\mu^\kappa)} \frac{G_\ep(x,y) \big(\nabla_d V_\mx^{p-1}\big)(y)}{|y|^{N-2}}\, \dy \\
&\ + O(1)\ep^{(N-2)p}\mu^{-\frac{Np}{p+1}} |x|^{2-(N-2)p} + o(1)\mu^{\frac{Np}{q+1}}.
\end{aligned}
\end{equation}
Applying Corollary \ref{rem C1} and the chain of inequalities $0 \le b^{p-1}-a^{p-1} \le (b-a)^{p-1}$ for $0 < a \le b$, we also compute
\begin{equation}\label{der est2}
\begin{aligned}
&\ \int_{\Omega_\ep \cap B(\xi,\mu^\kappa)} G_\ep(x,y) PV_\mx^{p-1}(y) \left\{\nabla_d\(PV_\mx-V_\mx\)\right\}(y)\dy \\
&=(1+o(1)) \(\frac{N}{p+1}\) d^{-1}\ep^{N-2}\mu^{-\frac{N}{p+1}}V(\tau) \int_{\Omega_\ep \cap B(\xi,\mu^\kappa)} \frac{G_\ep(x,y) V_\mx^{p-1}(y)}{|y|^{N-2}}\, \dy \\
&\ + O(1)\int_{\Omega_\ep \cap B(\xi,\mu^\kappa)} \frac{1}{|x-y|^{N-2}} \left[V_\mx^{p-1}(y)\mu^{\frac{N}{q+1}} \right.\\
& \hspace{125pt} \left. + \(\mu^{\frac{N(p-1)}{q+1}} + \mu^{-\frac{N(p-1)}{p+1}}\frac{\ep^{(N-2)(p-1)}}{|y|^{(N-2)(p-1)}}\) \(\mu^{\frac{N}{q+1}}+\mu^{-\frac{N}{p+1}}\frac{\ep^{N-2}}{|y|^{N-2}} \)\right]\dy \\
&= -(1+o(1)) \ep^{N-2}\nabla_d\(\mu^{-\frac{N}{p+1}}\) V(\tau)\int_{\Omega_\ep \cap B(\xi,\mu^\kappa)} \frac{G_\ep(x,y) V_\mx^{p-1}(y)}{|y|^{N-2}}\, \dy \\
&\ + O(1) \ep^{(N-2)p}\mu^{-\frac{Np}{p+1}} |x|^{2-(N-2)p} + o(1)\mu^{\frac{Np}{q+1}} + o(1)\mu^{-\frac{N}{q+1}}\ep^{N-2}|x|^{2-N}.
\end{aligned}
\end{equation}
To derive the second equality in \eqref{der est2}, we employed the computations in \eqref{inner est} and the identities
\begin{align*}
\mu^{-\frac{N(p-1)}{p+1}+\frac{N}{q+1}}\ep^{(N-2)(p-1)}\int_{\Omega_\ep \cap B(\xi,\mu^\kappa)}\frac{\dy}{|x-y|^{N-2}|y|^{(N-2)(p-1)}} &= O(1)\mu^{\frac{Np}{q+1}}\(\frac{\ep}{\mu}\)^{(N-2)(p-1)} \\
&= o(1)\mu^{\frac{Np}{q+1}} \quad (\textup{for } p > 1)
\end{align*}
and
\begin{align*}
\mu^{\frac{N(p-1)}{q+1}-\frac{N}{p+1}}\ep^{N-2} \int_{\Omega_\ep \cap B(\xi,\mu^\kappa)} \frac{\dy}{|x-y|^{N-2}|y|^{N-2}}
&=O(1)\mu^{\frac{N(p-1)}{q+1}-\frac{N}{p+1}}\ep^{N-2}|x|^{4-N} |\ln|x|| \\
&=\begin{cases} o(1)\mu^{-\frac{N}{q+1}}\ep^{N-2} |x|^{2-N} &\textup{if } \ep \le|x| \le \mu,\\
o(1)\mu^{\frac{Np}{q+1}} &\textup{if } \mu<|x|,
\end{cases}
\end{align*}
which are the consequences of \eqref{basic est o}.

Moreover, by the ($N$-dimensional) Leibniz-Reynolds transport theorem and
\begin{align*}
&\ \ep^{N-2}\mu^{-\frac{N}{p+1}} \int_{\pa B(\xi,\mu^\kappa)}\frac{G_\ep(x,y)V_\mx^{p-1}(y)}{|y|^{N-2}} {\bf v} \cdot \nu_y \textup{d}\sigma_y\\
&= O(1)\mu^{-\frac{N}{p+1}-\frac{N(p-1)}{p+1}+\kappa} \ep^{N-2}\int_{\pa B(\xi,\mu^\kappa)} \frac{\textup{d}\sigma_y}{|x-y|^{N-2}|y|^{N-2}|\mu^{-1}(y-\xi)|^{(N-2)(p-1)}}\\
&= O(1)\(\frac{\ep}{\mu}\)^{N-2}\mu^{\frac{Np}{q+1}+\kappa(2-(N-2)p)}=o(1)\mu^{\frac{Np}{q+1}}
\end{align*}
for $\kappa \in (0,1)$ small, where ${\bf v}$ is the velocity of the moving boundary $\pa B(\xi,\mu^\kappa)$ as $d$ varies,
$\nu_y$ is the outward unit normal vector on $\pa B(\xi,\mu^\kappa)$, and $\textup{d}\sigma_y$ is the surface measure whose variable is $y$, we have
\begin{equation}\label{der est3}
\begin{aligned}
&\ \nabla_d \left[\ep^{N-2}\mu^{-\frac{N}{p+1}} V(\tau) \int_{\Omega_\ep \cap B(\xi,\mu^\kappa)} \frac{G_\ep(x,y)V_\mx^{p-1}(y)}{|y|^{N-2}}\, \dy\right]\\
&= \ep^{N-2}\nabla_d\(\mu^{-\frac{N}{p+1}}\) V(\tau) \int_{\Omega_\ep \cap B(\xi,\mu^\kappa)} \frac{G_\ep(x,y)V_\mx^{p-1}(y)}{|y|^{N-2}}\, \dy \\
&\ +\ep^{N-2}\mu^{-\frac{N}{p+1}} V(\tau) \int_{\Omega_\ep \cap B(\xi,\mu^\kappa)}\frac{G_\ep(x,y)\(\nabla_d V_\mx^{p-1}\)(y)}{|y|^{N-2}}\, \dy+o(1)\mu^{\frac{Np}{q+1}}.
\end{aligned}
\end{equation}
As a result, combining \eqref{der est1}--\eqref{der est3}, we obtain that
\begin{equation}\label{inner c1}
\begin{aligned}
&\ p\int_{\Omega_\ep \cap B(\xi,\mu^\kappa)} G_\ep(x,y) \left[\(\nabla_d V_\mx\) \(PV_\mx^{p-1}-V_\mx^{p-1}\) + PV_\mx^{p-1} \left\{\nabla_d\(PV_\mx-V_\mx\)\right\}\right](y)\dy \\
&= -\nabla_d \left[\ep^{N-2}\mu^{-\frac{N}{p+1}} V(\tau) \int_{\Omega_\ep \cap B(\xi,\mu^\kappa)} \frac{G_\ep(x,y)V_\mx^{p-1}(y)}{|y|^{N-2}}\, \dy\right] \\
&\ + O(1) \ep^{(N-2)p}\mu^{-\frac{Np}{p+1}} |x|^{2-(N-2)p} + o(1)\left[\mu^{\frac{Np}{q+1}} + \mu^{-\frac{N}{q+1}}\ep^{N-2} \(|x|^{2-N}+\mca_{\ep,(d,\tau)}(x)\)\right].
\end{aligned}
\end{equation}

We next estimate the second term in the right-hand side of \eqref{der est}. Since
\begin{equation}\label{der for v}
\begin{aligned}
&\ (\nabla_d V_\mx)(y) \\
&= -d^{-1}\mu^{-\frac{N}{p+1}} \left[\frac{N}{p+1}V(\mu^{-1}(y-\xi)) + \mu^{-1}(y-\xi) \cdot (\nabla V)(\mu^{-1}(y-\xi)) +\tau \cdot (\nabla V)(\mu^{-1}(y-\xi))\right] \\
&= \(\frac{Nb_{N,p}}{q+1}\)d^{-1}\mu^{\frac{N}{q+1}} |y-\xi|^{2-N}+O(1)\mu^{\frac{N}{q+1}+\kappa_1}|y-\xi|^{2-N-\kappa_1} \quad (\textup{by } \eqref{decay for v} \textup{ and } \eqref{decay for grad v})
\end{aligned}
\end{equation}
for $y \in \Omega_\ep \setminus B(\xi,\mu^\kappa)$, we find from \eqref{v exp out} and \eqref{basic est} that
\begin{multline}\label{der outer1}
\int_{\Omega_\ep \setminus B(\xi,\mu^\kappa)} G_\ep(x,y)(\nabla_d V_\mx)(y) V_\mx^{p-1}(y)\dy \\
= \(\frac{Nb_{N,p}^p}{q+1}\) d^{-1}\mu^{\frac{Np}{q+1}} \int_{\Omega_\ep \setminus B(\xi,\mu^\kappa)} G_\ep(x,y) |y-\xi|^{(2-N)p}\, \dy + o(1)\mu^{\frac{Np}{q+1}}
\end{multline}
for $\kappa \in (0,1)$ small. Moreover, by \eqref{pv outer} and \eqref{der for v},
\[PV_\mx^{p-1}(y) =\(\frac{b_{N,p}}{\ga_N}\mu^{\frac{N}{q+1}}G(y,\xi)\)^{p-1}+O(1) \(\frac{\ep^{N-2}\mu^{-\frac{N}{p+1}}}{|y|^{N-2}}\)^{p-1} + O(1)\mu^{\frac{N(p-1)}{q+1}+\tilde{\kappa}}\]
and
\begin{align*}
(\nabla_d PV_\mx)(y) &= \frac{Nd^{-1}}{q+1} \(\frac{b_{N,p}}{\ga_N} \mu^{\frac{N}{q+1}}G(y,\xi)\) - \frac{b_{N,p}}{\ga_N} \mu^{\frac{N}{q+1}} \nabla_d (H(y,\xi))
+ O(1) \(\mu^{-\frac{N}{p+1}} \frac{\ep^{N-2}}{|y|^{N-2}} + \mu^{\frac{N}{q+1}+\tilde{\kappa}}\)
\end{align*}
for $y \in \Omega_\ep \setminus B(\xi,\mu^\kappa)$, where $\tilde{\kappa} \in (0,1)$ is a small number. Hence
\begin{equation}\label{der outer2}
\begin{aligned}
&\ \int_{\Omega_\ep \setminus B(\xi,\mu^\kappa)} G_\ep(x,y) (\nabla_d PV_\mx)(y) PV_\mx^{p-1}(y)\dy \\
&=\(\frac{b_{N,p}}{\ga_N}\)^p\mu^{\frac{Np}{q+1}} \int_{\Omega_\ep \setminus B(\xi,\mu^\kappa)} G_\ep(x,y) \left[\frac{Nd^{-1}}{q+1}G^p(y,\xi)-G^{p-1}(y,\xi)\nabla_d (H(y,\xi))\right]\dy + o(1)\mu^{\frac{Np}{q+1}} \\
&= \(\frac{b_{N,p}}{\ga_N}\)^p\mu^{\frac{Np}{q+1}}\int_{\Omega_\ep \setminus B(\xi,\mu^\kappa)} G_\ep(x,y)
\left[\frac{Nd^{-1}}{q+1}G^p(y,\xi) -p^{-1}\nabla_d \(\ga_N^p|y-\xi|^{(2-N)p}- G^p(y,\xi)\)\right]\dy \\
&\ + o(1)\mu^{\frac{Np}{q+1}}
\end{aligned}
\end{equation}
for $\kappa \in (0,1)$ small, where the second equality in \eqref{der outer2} follows from the estimate
\begin{align*}
&\ \int_{\Omega_\ep \setminus B(\xi,\mu^\kappa)} G_\ep(x,y)G^{p-1}(y,\xi)\nabla_d\(\ga_N|y-\xi|^{2-N}\)\dy\\
&=\int_{\Omega_\ep \setminus B(\xi,\mu^\kappa)} G_\ep(x,y) \(\ga_N^{p-1}|y-\xi|^{(2-N)(p-1)}+O(1)\) \nabla_d\(\ga_N|y-\xi|^{2-N}\)\dy\\
&=p^{-1}\int_{\Omega_\ep \setminus B(\xi,\mu^\kappa)} G_\ep(x,y) \nabla_d\(\ga_N^p|y-\xi|^{(2-N)p}\)\dy + O(1)\mu\int_{\Omega_\ep \setminus B(\xi,\mu^\kappa)}|x-y|^{2-N}|y-\xi|^{1-N}\, \dy\\
&=p^{-1}\int_{\Omega_\ep \setminus B(\xi,\mu^\kappa)} G_\ep(x,y) \nabla_d\(\ga_N^p|y-\xi|^{(2-N)p}\)\dy + O(1)\mu^{1-\kappa(N-3)} \quad \text{(by \eqref{basic est}}).
\end{align*}
Summing \eqref{der outer1} and \eqref{der outer2} shows
\begin{equation}\label{der exp}
\begin{aligned}
&\ p\int_{\Omega_\ep \setminus B(\xi,\mu^\kappa)} G_\ep(x,y) \left[PV_\mx^{p-1} \(\nabla_d PV_\mx\) - V_\mx^{p-1} \(\nabla_d V_\mx\)\right](y)\dy \\
&= \frac{Np}{q+1}\(\frac{b_{N,p}}{\ga_N}\)^p d^{-1}\mu^{\frac{Np}{q+1}} \int_{\Omega_\ep \setminus B(\xi,\mu^\kappa)} G_\ep(x,y) \(G^p(y,\xi)-\ga_N^p|y-\xi|^{(2-N)p}\)\dy\\
&\ +\(\frac{b_{N,p}}{\ga_N}\)^p\mu^{\frac{Np}{q+1}}\int_{\Omega_\ep \setminus B(\xi,\mu^\kappa)} G_\ep(x,y) \nabla_d \(G^p(y,\xi)-\ga_N^p|y-\xi|^{(2-N)p}\)\dy + o(1)\mu^{\frac{Np}{q+1}}.
\end{aligned}
\end{equation}
We note that
\begin{align*}
&\ p^{-1} \nabla_d \(G^p(y,\xi)-\ga_N^p|y-\xi|^{(2-N)p}\)\\
&= \ga_N^p(N-2)\mu_\ep |y-\xi|^{(2-N)p-2} \left[\(1+O(1)|y-\xi|^{N-2}\)^{p-1}\((y-\xi) \cdot \tau+O(1) |y-\xi|^N\)-(y-\xi)\cdot \tau\right]\\
&= \begin{cases}
O(1)\mu |y-\xi|^{(2-N)(p-1)-1} &\textup{if } |y-\xi|\le \mu^\kappa,\\
O(1)\mu |y-\xi|^{1-N} &\textup{if } |y-\xi|> \mu^\kappa.
\end{cases}
\end{align*}
Then arguing as in \eqref{g ep to g} yields
\begin{equation}\label{der exp1}
\int_{\Omega_\ep \setminus B(\xi,\mu^\kappa)} G_\ep(x,y) \nabla_d \(G^p(y,\xi)-\ga_N^p|y-\xi|^{(2-N)p}\)\dy
= (\nabla_d \varphi_\xi)(x)+O(1)\mu+O(1)\ep^{N-2}|x|^{2-N}
\end{equation}
where $\varphi_\xi$ is the solution to \eqref{varphi}. Thus, plugging \eqref{g ep to varp} and \eqref{der exp1} into
\eqref{der exp}, we see
\begin{multline}\label{outer c1}
p\int_{\Omega_\ep \setminus B(\xi,\mu^\kappa)} G_\ep(x,y) \left[PV_\mx^{p-1} \(\nabla_d PV_\mx\) - V_\mx^{p-1} \(\nabla_d V_\mx\)\right](y)\dy\\
=\(\frac{b_{N,p}}{\ga_N}\)^p\mu^{\frac{Np}{q+1}} \left[\frac{Np}{q+1}d^{-1}\varphi_\xi(x)+\nabla_d \varphi_\xi(x)+o(1)+O(1)\ep^{N-2}|x|^{2-N}\right].
\end{multline}

Finally, by putting \eqref{der est}, \eqref{inner c1}, \eqref{outer c1} and Corollary \ref{rem C1} together, we establish \eqref{R C1}.

\medskip
Also, \eqref{mca est} and the fact that
\begin{align*}
\int_{\pa B(\xi,\mu^\kappa)} \frac{\textup{d}\sigma_y}{|x-y|^{N-2} }
&=\mu^{\kappa(N-1)}\int_{\pa B(0,1)} \frac{\textup{d}\sigma_y}{|\mu^\kappa y-(x-\xi)|^{N-2}} \\
&=O(1)\mu^{\kappa(N-1)}\min\left\{\mu^{\kappa(2-N)},|x-\xi|^{2-N}\right\} \\
&=\begin{cases}
O(1) \mu^\kappa &\textup{ if } \ep< |x|<\mu^\kappa,\\
O(1)\mu^{\kappa(N-1)}|x|^{2-N} &\textup{ if }|x|\ge\mu^\kappa
\end{cases}
\end{align*}
imply \eqref{mca est C1} (where $\nabla_{(d,\tau)}$ is replaced with $\nabla_d$).
\end{proof}

\section{Estimate for the reduced energy}\label{sec red}
Recall the energy functional $I_\ep$ in \eqref{ene}, the set $\Lambda_\delta$ in \eqref{Lambda delta}, the relation \eqref{trans and scaling}
between $(d,\tau) \in \Lambda_\delta$ and $(\mu,\xi) \in (0,\infty) \times \R^N$, and the pair $(\mcp U_\mx,PV_\mx)$ defined by \eqref{PUPV eq} and \eqref{mcp U eq}.
Let $J_\ep:\Lambda_\delta \to \R$ be a reduced energy
\begin{equation}\label{red ene}
\begin{aligned}
J_\ep(d,\tau) &= I_\ep(\mcp U_\mx,PV_\mx) \\
&= \int_{\Omega_\ep} \left[\nabla \mcp U_\mx\cdot \nabla PV_\mx - \frac{1}{p+1}(PV_\mx)^{p+1} - \frac{1}{q+1}(\mcp U_\mx)^{q+1}\right].
\end{aligned}
\end{equation}

In the next two propositions, we expand $J_\ep$ in the $C^0$- and $C^1$-sense, respectively.
\begin{proposition}\label{C0 est}
Assume $N \ge 4$, $p \in (1,\frac{N-1}{N-2})$, and $\ep > 0$ small. Then it holds that
\begin{align}
J_\ep(d,\tau) &= \frac{2}{N}\int_{\R^N}U^{q+1} + \frac{\mu_\ep^{(N-2)p-2}}{p+1} \left[\(\frac{b_{N,p}}{\ga_N}\)^p d^{(N-2)p-2} \wth_0(0)\int_{\R^N}U^q + (p+1)\ga_N^{-1} d^{2-N}U(\tau)V(\tau)\right] \nonumber \\
&\ +o(1)\mu_\ep^{(N-2)p-2} \label{ene exp}
\end{align}
where $\wth_0$ is the function in \eqref{wth} with $y = 0$ and $C_\ep > 0$ in \eqref{o(1)} is chosen uniformly for $(d,\tau)\in \Lambda_\delta$.
\end{proposition}
\begin{proof}
Because
\begin{align*}
\int_{\Omega_\ep}(PV_\mx) ^{p+1} &= \int_{\Omega_\ep} (-\Delta \mcp U_\mx) PV_\mx = \int_{\Omega_\ep}\nabla \mcp U_\mx\cdot \nabla PV_\mx \\
&= \int_{\Omega_\ep} \mcp U_\mx(-\Delta PV_\mx) = \int_{\Omega_\ep}(\mcp U_\mx)U_\mx^q \quad \textup{(by \eqref{PUPV eq} and \eqref{mcp U eq})},
\end{align*}
we know that
\begin{equation}\label{ene exp pre}
\begin{aligned}
J_\ep(d,\tau) &= \int_{\Omega_\ep} \left[\frac{p}{p+1}(\mcp U_\mx)U_\mx^q-\frac{1}{q+1}(\mcp U_\mx)^{q+1}\right] \\
&= \(\frac{p}{p+1}-\frac{1}{q+1}\)\int_{\Omega_\ep}U_\mx^{q+1} + \(\frac{p}{p+1}-1\)
\int_{\Omega_\ep}(\mcp U_\mx-U_\mx)U_\mx^q \\
&\ -\frac{1}{q+1}\int_{\Omega_\ep} \left[(\mcp U_\mx)^{q+1}-U_\mx^{q+1}-(q+1)(\mcp U_\mx-U_\mx)U_\mx^q\right] \\
&= \frac{2}{N}\int_{\Omega_\ep}U_\mx^{q+1} -\frac{1}{p+1}\int_{\Omega_\ep}(\mcp U_\mx-U_\mx)U_\mx^q &(\textup{by \eqref{critical}}) \\
&\ -\frac{q}{2}\int_{\Omega_\ep} (t_x\mcp U_\mx+(1-t_x)U_\mx)^{q-1}(\mcp U_\mx-U_\mx)^2 \dx
\end{aligned}
\end{equation}
where $t_x \in [0,1]$.

For the first and last terms on the rightmost side of \eqref{ene exp pre}, we have
\begin{equation}\label{eng exp1}
\begin{aligned}
\int_{\Omega_\ep} U_\mx^{q+1} &= \int_{\R^N}U^{q+1} + \int_{B(-\tau,\frac{\ep}{\mu})}U^{q+1} + \int_{\R^N\setminus \(\mu^{-1}(\Omega-\xi)\)}U^{q+1} \\
&= \int_{\R^N}U^{q+1} + O(1)\left[\(\frac{\ep}{\mu}\)^N + \mu^{Np}\right] = \int_{\R^N}U^{q+1} + o(1)\mu_\ep^{(N-2)p-2} \quad \text{(by \eqref{simple rel})}
\end{aligned}
\end{equation}
and
\begin{equation}\label{eng exp20}
0\le \int_{\Omega_\ep} (t_x\mcp U_\mx+(1-t_x)U_\mx)^{q-1}(\mcp U_\mx-U_\mx)^2 \dx \le \int_{\Omega_\ep}U_\mx^{q-1}(\mcp U_\mx-U_\mx)^2,
\end{equation}
for $0 \le \mcp U_\mx \le U_\mx$ in $\Omega_\ep$. If we define
\begin{multline}\label{Q mu}
Q_\mx(x) = \int_{\Omega_\ep}U_\mx^{q-1} \left[\mu^{\frac{2Np}{q+1}} + \(\ep^{N-2}\mu^{-\frac{N}{q+1}}\)^2 {|x|^{2(2-N)}} \right. \\
\left. + \(\ep^{N-2}\mu^{-\frac{N}{q+1}}\)^2 \mca_{\ep,(d,\tau)}^2(x) + \(\ep^{(N-2)p}\mu^{-\frac{Np}{p+1}}\)^2 |x|^{2(2-(N-2)p)}\right] \quad \textup{for } x \in \Omega_\ep,
\end{multline}
then \eqref{eng exp20}, \eqref{mcp U} and Lemma \ref{high est2} give
\begin{equation}\label{eng exp2}
\int_{\Omega_\ep} (t_x\mcp U_\mx+(1-t_x)U_\mx)^{q-1}(\mcp U_\mx-U_\mx)^2 \dx = O(1)Q_\mx =o(1)\mu_\ep^{(N-2)p-2}.
\end{equation}

For the second term in the rightmost side of \eqref{ene exp pre}, we see from \eqref{mcp U} that
\begin{equation}\label{second term est}
\begin{aligned}
&\ \int_{\Omega_\ep} (\mcp U_\mx-U_\mx)U_\mx^q \\
&=\int_{\Omega_\ep} U_\mx^q(x) \left[-(1+o(1))\(\frac{b_{N,p}}{\ga_N}\)^p\mu^{\frac{Np}{q+1}}\wth_\xi(x) - (1+o(1))\ep^{N-2}\mu^{-\frac{N}{q+1}}U(\tau) {|x|^{2-N}} \right. \\
&\hspace{65pt} \left. -(1+o(1))\ep^{N-2}\mu^{-\frac{N}{q+1}} pV(\tau)\mca_{\ep,(d,\tau)}(x) + O(1)\ep^{(N-2)p}\mu^{-\frac{Np}{p+1}} |x|^{2-(N-2)p} \right] \dx.
\end{aligned}
\end{equation}
Applying Lemma \ref{wth reg}, \eqref{limit system}, \eqref{decay for u}, and Green's representation formula, we compute
\begin{equation}\label{second term est1}
\begin{aligned}
\mu^{\frac{Np}{q+1}} \int_{\Omega_\ep} U_\mx^q \wth_\xi
&=\mu^{\frac{Np}{q+1}}\wth_\xi(\xi)\int_{\Omega_\ep} U_\mx^q + O(1)\mu_\ep^{\frac{Np}{q+1}} \int_{\Omega_\ep} U_\mx^q(x)|x-\xi|\dx\\
&=\mu^{(N-2)p-2}\wth_0(0)\int_{\R^N}U^q + o(1)\mu_\ep^{(N-2)p-2}
\end{aligned}
\end{equation}
and
\begin{equation}\label{second term est3}
\begin{aligned}
&\ \ep^{N-2}\mu^{-\frac{N}{q+1}} \int_{\Omega_\ep} U_\mx^q(x)|x|^{2-N}\, \dx \\
&=\(\frac{\ep}{\mu}\)^{N-2} \int_{\mu^{-1}(\Omega_\ep-\xi)}\frac{(-\Delta V)(x)}{|x+\tau|^{N-2}}\, \dx \\
&=\(\frac{\ep}{\mu}\)^{N-2} \left[\ga_N^{-1}V(\tau) + O(1)\int_{\R^N \setminus B(0,C\mu^{-1})} \frac{\dx}{|x|^{N-2+((N-2)p-2)q}} + O(1) \(\frac{\ep}{\mu}\)^2\right] \\
&= \(\frac{\ep}{\mu}\)^{N-2}\(\ga_N^{-1}V(\tau)+o(1)\).
\end{aligned}
\end{equation}
Moreover, Lemma \ref{basic est 2}, \eqref{simple rel} and $p > 1$ imply that
\begin{equation}\label{second term est2}
\ep^{(N-2)p}\mu^{-\frac{Np}{p+1}} \int_{\Omega_\ep} U_\mx^q(x)|x|^{2-(N-2)p}\, \dx = O(1)\(\frac{\ep}{\mu}\)^{(N-2)p} = o(1)\mu_\ep^{(N-2)p-2}.
\end{equation}
Next, using \eqref{A mu}, we observe
\begin{equation}\label{a mu est}
\begin{aligned}
&\ \ep^{N-2}\mu^{-\frac{N}{q+1}} \int_{\Omega_\ep} U_\mx^q(x) \mca_{\ep,(d,\tau)}\, \dx \\
&=\ep^{N-2} \int_{\mu^{-1}(\Omega_\ep-\xi)} \int_{B(0,\mu^{\kappa-1})\setminus B(-\tau,\frac{\ep}{\mu})} \frac{G_\ep(\mu x+\xi,\mu y+\xi)U^q(x) V^{p-1}(y)}{|y+\tau|^{N-2}}\, \dy\dx \\
&=\(\frac{\ep}{\mu}\)^{N-2} \ga_N\int_{\mu^{-1}(\Omega_\ep-\xi)} \int_{B(0,\mu^{\kappa-1})\setminus B(-\tau,\frac{\ep}{\mu})} \frac{U^q(x) V^{p-1}(y)}{|x-y|^{N-2} |y+\tau|^{N-2}}\, \dy\dx \\
&\ - \ep^{N-2} \int_{\mu^{-1}(\Omega_\ep-\xi)} \int_{B(0,\mu^{\kappa-1})\setminus B(-\tau,\frac{\ep}{\mu})} \frac{H_\ep(\mu x+\xi,\mu y+\xi)U^q(x) V^{p-1}(y)}{|y+\tau|^{N-2}}\, \dy\dx \\
&=: \textup{(I)}+\textup{(II)}.
\end{aligned}
\end{equation}
Then Lemma \ref{reg part ene est} allows us to deduce
\begin{equation}\label{a mu est1}
\textup{(II)} = o(1)\mu_\ep^{(N-2)p-2}.
\end{equation}
Let us estimate $\textup{(I)}$. By virtue of Fubini's theorem and \eqref{limit system},
\begin{align*}
\int_{\R^N} \int_{\R^N} \frac{U^q(x) V^{p-1}(y)}{|x-y|^{N-2} |y+\tau|^{N-2}}\, \dy\dx &= \int_{\R^N} \frac{V^{p-1}(y)}{|y+\tau|^{N-2}}\int_{\R^N} \frac{U^q(x)}{|x-y|^{N-2}}\, \dx\dy\\
&= \ga_N^{-1}\int_{\R^N} \frac{V^p(y)}{|y+\tau|^{N-2}} \dy = \ga_N^{-2}U(\tau),
\end{align*}
\begin{align*}
&\ \int_{\left[\R^N\setminus B(0,C\mu^{-1})\right] \cup B(-\tau,\frac{\ep}{\mu})} \int_{\R^N} \frac{U^q(x) V^{p-1}(y)}{|x-y|^{N-2} |y+\tau|^{N-2}}\, \dy\dx\\
&= O(1)\int_{\R^N} \frac{V^{p-1}(y)}{|y+\tau|^{N-2}} \left[\int_{\R^N\setminus B(0,C\mu^{-1})} \frac{\dx}{|x-y|^{N-2}|x|^{((N-2)p-2)q}} + \int_{B(-\tau,\frac{\ep}{\mu})} \frac{\dx}{|x-y|^{N-2}} \right] \dy\\
&= O(1) \left[\mu^{((N-2)p-2)q-2}\int_{\R^N} \frac{V^{p-1}(y)}{|y+\tau|^{N-2}} \int_{\R^N\setminus B(0,1)} \frac{\dx}{|x-\mu y|^{N-2}|x|^{((N-2)p-2)q}}\, \dy \right. \\
&\hspace{150pt} \left. + \(\frac{\ep}{\mu}\)^N \int_{\R^N} \frac{V^{p-1}(y)}{|y+\tau|^{N-2}} \frac{\dy}{\big(\frac{\ep}{\mu}\big)^{N-2}+|y+\tau|^{N-2}}\right] \quad \textup{(by \eqref{interior es})}\\
&= O(1) \left[\mu^{((N-2)p-2)q-2}\int_{\R^N} \frac{V^{p-1}(y)}{|y+\tau|^{N-2}}\min\left\{1, |\mu y|^{2-N}\right\} \dy \right. \\
&\hspace{115pt} \left. + \(\frac{\ep}{\mu}\)^N \int_{\R^N} \frac{V^{p-1}(y-\tau)}{|y|^{N-2}} \min\left\{\(\frac{\ep}{\mu}\)^{2-N}, |y|^{2-N}\right\} \dy\right] \\
& \hspace{245pt} (\textup{by } \eqref{basic est whole} \textup{ and } ((N-2)p-2)q > N+2)\\
&= O(1) \left[\mu^{((N-2)p-2)(q+1)-N} + \(\frac{\ep}{\mu}\)^4 \(1+\left|\ln \frac{\ep}{\mu}\right|\mone_{N=4}\)\right] = o(1)
\end{align*}
and
\begin{align*}
&\ \int_{\R^N} \int_{[\R^N\setminus B(0,\mu^{\kappa-1})]\cup B(-\tau,\frac{\ep}{\mu})} \frac{U^q(x) V^{p-1}(y)}{|x-y|^{N-2} |y+\tau|^{N-2}}\, \dy\dx \\
&= \int_{[\R^N\setminus B(0,\mu^{\kappa-1})]\cup B(-\tau,\frac{\ep}{\mu})} \frac{V^{p-1}(y)}{|y+\tau|^{N-2}}\int_{\R^N} \frac{U^q(x)}{|x-y|^{N-2}}\, \dx\dy \\
&= \ga_N^{-1} \int_{[\R^N\setminus B(0,\mu^{\kappa-1})]\cup B(-\tau,\frac{\ep}{\mu})} \frac{V^p(y)}{|y+\tau|^{N-2}}\, \dy \\
&= O(1) \left[\int_{\R^N\setminus B(0,\mu^{\kappa-1})} \frac{\dy}{|y|^{(N-2)(p+1)}} + \int_{B(0,\frac{\ep}{\mu})}\frac{\dy}{|y|^{N-2}}\right] \\
&= O(1) \left[\mu^{(1-\kappa)((N-2)p-2)}+\(\frac{\ep}{\mu}\)^2\right]=o(1).
\end{align*}
Consequently,
\begin{equation}\label{a mu est2}
\textup{(I)} = (1+o(1))\(\frac{\ep}{\mu}\)^{N-2}\ga_N^{-1} U(\tau).
\end{equation}
In view of \eqref{a mu est}--\eqref{a mu est2}, we conclude that
\begin{equation}\label{second term est4}
\ep^{N-2}\mu^{-\frac{N}{q+1}}V(\tau) \int_\Omega U_\mx^q \mca_{\ep,(d,\tau)} = (1+o(1))\(\frac{\ep}{\mu_\ep}\)^{N-2}d^{2-N}\ga_N^{-1} U(\tau)V(\tau).
\end{equation}
Combining \eqref{second term est}--\eqref{second term est2} and \eqref{second term est4}, we establish
\begin{equation}\label{eng exp3}
\begin{aligned}
\int_{\Omega_\ep} (\mcp U_\mx-U_\mx)U_\mx^q &=- \mu^{(N-2)p-2} \(\frac{b_{N,p}}{\ga_N}\)^p \wth_0(0)\int_{\R^N}U^q \\
&\ - (p+1)\ga_N^{-1} \mu_\ep^{(N-2)p-2}d^{2-N} U(\tau)V(\tau) + o(1)\mu_\ep^{(N-2)p-2}.
\end{aligned}
\end{equation}
Thus, from \eqref{ene exp pre}, \eqref{eng exp1}, \eqref{eng exp20}, \eqref{eng exp2} and \eqref{eng exp3}, we arrive at \eqref{ene exp}.
\end{proof}

\begin{proposition}\label{C1 est}
Under the hypotheses of Proposition \ref{C0 est}, we have
\begin{align}
\nabla_{(d,\tau)} J_\ep(d,\tau) &= \(\frac{\mu_\ep^{(N-2)p-2}}{p+1}\) \nabla_{(d,\tau)} \left[\(\frac{b_{N,p}}{\ga_N}\)^p d^{(N-2)p-2} \wth_0(0)\int_{\R^N}U^q + (p+1)\ga_N^{-1}d^{2-N} U(\tau)V(\tau)\right] \nonumber \\
&\ +o(1)\mu_\ep^{(N-2)p-2} \label{ene exp C1}
\end{align}
where $C_\ep > 0$ in \eqref{o(1)} is chosen uniformly for $(d,\tau)\in \Lambda_\delta$.
\end{proposition}
\begin{proof}
Here, we only consider the differentiation of the map $J_\ep$ with respect to the $d$-variable.

By \eqref{red ene}, \eqref{PUPV eq}, \eqref{mcp U eq}, \eqref{mcp U}--\eqref{mca est} and \eqref{R C1}--\eqref{mca est C1}, we have
\begin{align}
\nabla_d J_\ep(d,\tau) &= \int_{\Omega_\ep} \left[U_\mx^q - (\mcp U_\mx)^q\right]  \nabla_d \mcp U_\mx \nonumber \\
&= \(\frac{b_{N,p}}{\ga_N}\)^p\mu^{\frac{Np}{q+1}} \int_{\Omega_\ep} \(qU_\mx^{q-1} \nabla_d U_\mx\) \wth_\xi
+ \ep^{N-2}\mu^{-\frac{N}{q+1}} U(\tau) \int_{\Omega_\ep} \frac{\(qU_\mx^{q-1}\nabla_d U_\mx\)(x)}{|x|^{N-2}}\, \dx \nonumber \\
&\ + \ep^{N-2}\mu^{-\frac{N}{q+1}} p V(\tau) \int_{\Omega_\ep} \(qU_\mx^{q-1}\nabla_d U_\mx\)(x)\mca_{\ep,(d,\tau)}(x) \dx + o(1)\mu_\ep^{(N-2)p-2} \label{ene C1 est1} \\
&=: \textup{(III)} + \textup{(IV)} + \textup{(V)} + o(1)\mu_\ep^{(N-2)p-2}. \nonumber
\end{align}
We will estimate each of (III), (IV) and (V).

Firstly, it holds that
\[\int_{\Omega_\ep} \(U_\mx^{q-1} \nabla_d U_\mx\) \wth_\xi = - \mu^{N \over q+1} d^{-1} \wth_0(0) \left[\frac{N}{q+1} \int_{\R^N} U^q + \int_{\R^N} (x \cdot \nabla U(x)) U^{q-1}(x) \dx\right] + o(1)\mu^{N \over q+1}.\]
Combining this and the equality
\[q\int_0^{\infty} r^N U^{q-1}(r) (\pa_rU)(r) \textup{d}r = -N\int_0^{\infty}r^{N-1}U^q(r) \textup{d}r \quad (\text{where } r = |x| \text{ and } U(r) = U(x)),\]
we find
\begin{equation}\label{ene C1 est2}
\textup{(III)} = \mu_\ep^{(N-2)p-2} \(\frac{b_{N,p}}{\ga_N}\)^p \(\frac{N}{q+1}\) d^{(N-2)p-3} \wth_0(0) \int_{\R^N}U^q + o(1)\mu_\ep^{(N-2)p-2}.
\end{equation}

Secondly, arguing similarly to \eqref{second term est3}, we see
\begin{align*}
\int_{\Omega_\ep} \frac{\(qU_\mx^{q-1}\nabla_d U_\mx\)(x)}{|x|^{N-2}}\, \dx &= -\ga_N^{-1} \mu^{2-\frac{qN}{q+1}} d^{-1} \left[\Phi_{1,0}^0(-\tau) + \tau \cdot (\nabla V)(-\tau) + o(1)\right] \\
&= -\(\frac{N}{p+1}\) \ga_N^{-1} \mu^{2-\frac{qN}{q+1}} d^{-1}\(V(\tau)+o(1)\) \quad (\textup{by \eqref{PsPhmxl10}}),
\end{align*}
from which we deduce
\begin{equation}\label{ene C1 est3}
\textup{(IV)} = -\(\frac{N}{p+1}\) \ga_N^{-1} \(\frac{\ep}{\mu}\)^{N-2} d^{-1} U(\tau)V(\tau) + o(1)\mu_\ep^{(N-2)p-2}.
\end{equation}

Lastly, as in \eqref{second term est4}, we obtain
\begin{align*}
\int_{\Omega_\ep} \(qU_\mx^{q-1}\nabla_d U_\mx\)(x)\mca_{\ep,(d,\tau)}(x) \dx &= -\frac{1}{p}\,\ga_N^{-1} \mu^{-\frac{N}{p+1}} d^{-1} \left[\Psi_{1,0}^0(-\tau) + \tau \cdot (\nabla U)(-\tau) + o(1)\right] \\
&= -\frac{1}{p}\(\frac{N}{q+1}\)\ga_N^{-1} \mu^{-\frac{N}{p+1}} d^{-1}\(U(\tau)+o(1)\) \quad (\textup{by \eqref{PsPhmxl10}}).
\end{align*}
Hence
\begin{equation}\label{ene C1 est4}
\textup{(V)} = -\(\frac{N}{q+1}\)\ga_N^{-1} \(\frac{\ep}{\mu}\)^{N-2} d^{-1} U(\tau)V(\tau) + o(1)\mu_\ep^{(N-2)p-2}.
\end{equation}

Putting \eqref{ene C1 est1}--\eqref{ene C1 est4} and \eqref{critical} together, we establish \eqref{ene exp C1} (where $\nabla_{(d,\tau)}$ is replaced with $\nabla_d$).
\end{proof}

\section{Reduction process}\label{sec LSred}
In this section, we outline the main steps of the Lyapunov-Schmidt reduction.

\subsection{Reformulation of the problem}
We consider a map $\mci^*_\ep:L^{\frac{p+1}{p}}(\Omega_\ep)\times L^{\frac{q+1}{q}}(\Omega_\ep) \to X_{p,q}(\Omega_\ep)$ given by $\mci^*_\ep(f,g)=(u,v)$, where
$$
\begin{cases}
-\Delta u=f &\textup{in } \Omega_\ep,\\
-\Delta v=g &\textup{in } \Omega_\ep,\\
u=v=0 &\textup{on } \pa \Omega_\ep
\end{cases}
\quad \textup{or equivalently,} \quad
\begin{cases}
\displaystyle u(x)=\int_{\Omega_\ep} G_\ep(x,y)f(y)\dy,\\
\displaystyle v(x)=\int_{\Omega_\ep} G_\ep(x,y)g(y)\dy
\end{cases}
\textup{for } x \in \Omega_\ep.$$
Then the operator norm of $\mci^*_\ep$ is uniformly bounded in $\ep > 0$ small, and system \eqref{main system} is rewritten as
$$
(u,v)=\mci^*_\ep\(|v|^{p-1}v,|u|^{q-1}u\).
$$

Given a pair $(p,q)$ and a small $\ep > 0$ fixed, we write $X_\ep = X_{p,q}(\Omega_{\ep})$ for brevity.
Also, we recall the functions $(U_\mx,V_\mx)$ and $(\Psi_\mx^l,\Phi_\mx^l)$ in \eqref{UVmx} and \eqref{PsPhmxl}.
For $(d,\tau) \in \Lambda_\delta$, let $Y_{\ep,(d,\tau)}$ and $Z_{\ep,(d,\tau)}$ be the subspaces of $X_\ep$ defined as
\[Y_{\ep,(d,\tau)} = \textup{span} \left\{\(P\Psi_\mx^l,P\Phi_\mx^l\): l = 0, \ldots, N \right\}\]
and
\[Z_{\ep,(d,\tau)} = \left\{(\psi,\phi) \in X_\ep: \int_{\Omega_\ep} \(p V_\mx^{p-1} \Phi_\mx^l \phi + q U_\mx^{q-1} \Psi_\mx^l \psi\) = 0 \textup{ for } l = 0, \ldots, N \right\}\]
where $(P\Psi_\mx^l,P\Phi_\mx^l)$ is the unique smooth solution to the system
\begin{equation}\label{PPsPPhmxl eq}
\begin{cases}
-\Delta P\Psi_\mx^l = pV_\mx^{p-1}\Phi_\mx^l &\textup{in } \Omega_\ep,\\
-\Delta P\Phi_\mx^l = qU_\mx^{q-1}\Psi_\mx^l &\textup{in } \Omega_\ep,\\
P\Psi_\mx^l = P\Phi_\mx^l = 0 &\textup{on } \pa \Omega_\ep.
\end{cases}
\end{equation}
Then, arguing as in \cite[Lemma 3.1]{KP}, one sees that $Y_{\ep,(d,\tau)}$ and $Z_{\ep,(d,\tau)}$ are topological complements of each other.
Besides, if $\Pi_{\ep,(d,\tau)}: X_\ep \to Y_{\ep,(d,\tau)}$ is the linear operator defined by
\begin{equation}\label{Pi}
\Pi_{\ep,(d,\tau)}(\psi,\phi) = \sum_{l=0}^N c_{\ep,l} \(P\Psi_\mx^l,P\Phi_\mx^l\)
\end{equation}
where the coefficients $c_{\ep,l}$'s are defined by the relation
\begin{multline*}
\sum_{m=0}^N c_{\ep,m} \int_{\Omega_\ep} \(p V_\mx^{p-1}\Phi_\mx^l P\Phi_\mx^m + q U_\mx^{q-1} \Psi_\mx^l P\Psi_\mx^m\)
= \int_{\Omega_\ep} \(p V_\mx^{p-1} \Phi_\mx^l \phi + q U_\mx^{q-1} \Psi_\mx^l \psi\) \\
\textup{for } l = 0, \ldots, N,
\end{multline*}
then a slight modification of the proof of \cite[Corollary 3.2]{KP} shows that $\Pi_{\ep,(d,\tau)}$ is well-defined and its operator norm is uniformly bounded in $\ep > 0$ small and $(d,\tau) \in \Lambda_\delta$.

To build a solution to \eqref{main system}, we shall look for $(d,\tau) \in \Lambda_\delta$ such that $(\psi,\phi) = (\psi_{\ep,(d,\tau)},\phi_{\ep,(d,\tau)}) \in Z_{\ep,(d,\tau)}$ satisfies
\begin{enumerate}
\item the auxiliary equation:
\begin{multline}\label{aux eq}
\(\textup{Id}_{X_\ep}-\Pi_{\ep,(d,\tau)}\) \left[\(\mcp U_\mx+\psi, PV_\mx+\phi\) \right.\\
\left. - \mci^*_\ep\(|PV_\mx+\phi|^{p-1} \(PV_\mx+\phi\), |\mcp U_\mx+\psi|^{q-1} \(\mcp U_\mx+\psi\)\)\right] = 0;
\end{multline}
\item the bifurcation equation:
\[\Pi_{\ep,(d,\tau)} \left[\(\mcp U_\mx+\psi, PV_\mx+\phi\) - \mci^*_\ep\(|PV_\mx+\phi|^{p-1} \(PV_\mx+\phi\), |\mcp U_\mx+\psi|^{q-1} \(\mcp U_\mx+\psi\)\)\right] = 0\]
\end{enumerate}
where $\textup{Id}_{X_\ep}$ is the identity operator on $X_\ep$.

Let $L_{\ep,(d,\tau)}: Z_{\ep,(d,\tau)} \to Z_{\ep,(d,\tau)}$ be a bounded linear operator
\begin{equation}\label{lin op}
L_{\ep,(d,\tau)}(\psi,\phi) = (\psi,\phi) - \(\textup{Id}_{X_\ep}-\Pi_{\ep,(d,\tau)}\) \left[\mci^\ast_\ep\(p(PV_\mx)^{p-1}\phi, q(\mcp U_\mx)^{q-1}\psi\)\right]
\end{equation}
for all $(\psi,\phi) \in Z_{\ep,(d,\tau)}$. We also define an error term
\begin{equation}\label{error}
\mce_{\ep,(d,\tau)} = (\mcp U_\mx,PV_\mx) - \mci^*_\ep\((PV_\mx)^p,(\mcp U_\mx)^q\)
\end{equation}
and a nonlinear operator $\mbN_{\ep,(d,\tau)}: Z_{\ep,(d,\tau)} \to L^{\frac{p+1}{p}}(\Omega_\ep)\times L^{\frac{q+1}{q}}(\Omega_\ep)$ by
\begin{align*}
\mbN_{\ep,(d,\tau)}(\psi,\phi) = & \(|PV_\mx+\phi|^{p-1}\(PV_\mx+\phi\) - (PV_\mx)^p - p(PV_\mx)^{p-1}\phi, \right. \\
&\ \left. |\mcp U_\mx+\psi|^{q-1}\(\mcp U_\mx+\psi\) - (\mcp U_\mx)^q - q(\mcp U_\mx)^{q-1}\psi\).
\end{align*}
Then the auxiliary equation is rewritten as
\begin{equation}\label{aux eq1}
\(\textup{Id}_{X_\ep}-\Pi_{\ep,(d,\tau)}\) \left[L_{\ep,(d,\tau)}(\psi,\phi) + \mce_{\ep,(d,\tau)} - \mci^*_\ep(\mbN_{\ep,(d,\tau)}(\psi,\phi))\right] = 0.
\end{equation}

\subsection{Error estimates}
This subsection is devoted to estimating the error term $\mce_{\ep,(d,\tau)}$ in \eqref{error}.
\begin{proposition}
Let $N \ge 4$, $p \in (1,\frac{N}{N-2})$, and $\ep > 0$ small. Then it holds that
\begin{equation}\label{error est}
\|\mce_{\ep,(d,\tau)}\|_{X_\ep} = O(1)\left[\mu^{(N-2)p-2}|\ln \mu|^{\frac{q}{q+1}}+ \(\frac{\ep}{\mu}\)^{\frac{Nq}{q+1}} \left|\ln\(\frac{\ep}{\mu}\)\right|^{\frac{q}{q+1}} \right]
\end{equation}
uniformly in $(d,\tau) \in \Lambda_\delta$.
\end{proposition}
\begin{proof}
By \eqref{PUPV eq}, \eqref{mcp U eq} and \eqref{mca est}, we have
\begin{equation}\label{error est2}
\|\mce_{\ep,(d,\tau)}\|_{X_\ep} = \left\|\mci^*_\ep\(0,U_\mx^q-(\mcp U_\mx)^q\)\right\|_{X_\ep} = O(1) \left\|U_\mx^q-(\mcp U_\mx)^q\right\|_{L^{\frac{q+1}{q}}(\Omega_\ep)}
\end{equation}
and
\begin{equation}\label{error est3}
\begin{aligned}
\ep^{N-2}\mu^{-\frac{N}{q+1}} \mca_{\ep,(d,\tau)}(x) &= O(1)\ep^{N-2}\mu^{-\frac{N}{q+1}+(N-2)p-N}|x|^{2-(N-2)p} \\
&= \(\frac{\ep}{\mu}\)^{N-2}\mu^{\frac{Np}{q+1}}{|x|^{2-(N-2)p}}
\end{aligned}
\end{equation}
for $x \in \Omega_\ep$. We observe from \eqref{mcp U}, \eqref{error est3} and the estimate $\ep^{(N-2)p}\mu^{-\frac{Np}{p+1}} 
= o(1) (\frac{\ep}{\mu})^{N-2}\mu^{\frac{Np}{q+1}}$ that
\begin{align*}
&\ \left|U_\mx^q(x)-\mcp U_\mx^q(x)\right|^{\frac{q+1}{q}}\\
&\le O(1) \left|U_\mx^{q-1}(x) \(\mu^{\frac{Np}{q+1}} + \ep^{N-2}\mu^{-\frac{N}{q+1}}\(|x|^{2-N} + \mca_{\ep,(d,\tau)}(x)\) + \ep^{(N-2)p}\mu^{-\frac{Np}{p+1}}|x|^{2-(N-2)p}\)\right|^{\frac{q+1}{q}}\\
&\ + \left|\mu^{\frac{Np}{q+1}} + \ep^{N-2}\mu^{-\frac{N}{q+1}}\(|x|^{2-N} + \mca_{\ep,(d,\tau)}(x)\) + \ep^{(N-2)p}\mu^{-\frac{Np}{p+1}}|x|^{2-(N-2)p}\right|^{q+1}\\
&\le O(1)U_\mx^{\frac{q^2-1}{q}}(x) \left[\mu^{\frac{Np}{q}} + \(\frac{\ep}{\mu}\)^{\frac{(N-2)(q+1)}{q}}\mu^{\frac{N(q+1)}{(p+1)q}} |x|^{\frac{(2-N)(q+1)}{q}}
+ \(\frac{\ep}{\mu}\)^{\frac{(N-2)(q+1)}{q}}\mu^{\frac{Np}{q}} |x|^{-\frac{N(p+1)}{q}}\right]\\
&\ + O(1) \left[\mu^{Np} + \ep^{(N-2)(q+1)}\mu^{-N} |x|^{(2-N)(q+1)} + \(\frac{\ep}{\mu}\)^{(N-2)(q+1)}\mu^{Np} |x|^{-N(p+1)}\right]
\end{align*}
for $x \in \Omega_\ep$.

On the other hand, applying Lemma \ref{basic est 2}, $(N-2)q-2-Np = (q+1)(N-(N-2)p)$ and
$$
((N-2)p-2)\frac{q^2-1}{q} + \frac{(N-2)(q+1)}{q} > ((N-2)p-2)\frac{q^2-1}{q}+\frac{((N-2)p-2)(q+1)}{q} > N,
$$
we deduce
\[\mu^{\frac{Np}{q}}\int_{\Omega_\ep}U_\mx^{\frac{q^2-1}{q}}(x)\dx = O(1)\mu^{\frac{Np}{q}} \(\mu^{\frac{N}{q}}|\ln \mu|+\mu^{\frac{Np(q-1)}{q}}\) = O(1)\(\mu^{\frac{N(p+1)}{q}}|\ln \mu|+\mu^{Np}\),\]
\[\(\frac{\ep}{\mu}\)^{\frac{(N-2)(q+1)}{q}}\mu^{\frac{N(q+1)}{(p+1)q}} \int_{\Omega_\ep} U_\mx^{\frac{q^2-1}{q}}(x) |x|^{\frac{(2-N)(q+1)}{q}} \dx
= O(1) \left[\(\frac{\ep}{\mu}\)^N\left|\ln\(\frac{\ep}{\mu}\)\right| + \(\frac{\ep}{\mu}\)^{\frac{(N-2)(q+1)}{q}}\right],\]
and
\[\(\frac{\ep}{\mu}\)^{\frac{(N-2)(q+1)}{q}}\mu^{\frac{Np}{q}} \int_{\Omega_\ep}U_\mx^{\frac{q^2-1}{q}}(x)|x|^{-\frac{N(p+1)}{q}}\, \dx
= \(\frac{\ep}{\mu}\)^{N+\frac{q+1}{q}(N-(N-2)p)}\left|\ln\(\frac{\ep}{\mu}\)\right| + \(\frac{\ep}{\mu}\)^{\frac{(N-2)(q+1)}{q}}.\]
We also note that
\begin{multline*}
\int_{\Omega_\ep} \left[\mu^{Np} + \ep^{(N-2)(q+1)}\mu^{-N} |x|^{(2-N)(q+1)} + \(\frac{\ep}{\mu}\)^{(N-2)(q+1)}\mu^{Np} |x|^{-N(p+1)}\right] \dx\\
=O(1)\left[\mu^{Np} + \(\frac{\ep}{\mu}\)^N\left\{1+\(\frac{\ep}{\mu}\)^{(q+1)(N-(N-2)p)}\right\}\right] = O(1)\left[\mu^{Np} + \(\frac{\ep}{\mu}\)^N\right].
\end{multline*}

Therefore, we arrive at
\begin{align*}
\left\|U_\mx^q-\mcp U_\mx^q\right\|_{L^{\frac{q+1}{q}}(\Omega_\ep)}
&= O(1)\left[\mu^{\frac{N(p+1)}{q}}|\ln \mu| + \mu^{Np} + \(\frac{\ep}{\mu}\)^N \left|\ln\(\frac{\ep}{\mu}\)\right| + \(\frac{\ep}{\mu}\)^{\frac{(N-2)(q+1)}{q}}\right]^\frac{q}{q+1}\\
&= O(1)\left[\mu^{(N-2)p-2}|\ln \mu|^{\frac{q}{q+1}}+ \(\frac{\ep}{\mu}\)^{\frac{Nq}{q+1}} \left|\ln\(\frac{\ep}{\mu}\)\right|^{\frac{q}{q+1}} \right] \quad \textup{(by \eqref{eq main1} and \eqref{simple rel})},
\end{align*}
which together with \eqref{error est2} yields \eqref{error est}.
\end{proof}

\subsection{Linear theory}
Recall the operator $L_{\ep,(d,\tau)}$ in \eqref{lin op}. Here, we examine the unique solvability of the linear equation
\begin{equation}\label{lin eq}
L_{\ep,(d,\tau)}(\psi,\phi) = (h_1,h_2)
\end{equation}
for any given $(h_1,h_2) \in Z_{\ep,(d,\tau)}$.

Employing Proposition \ref{nondeg}, we obtain the following result. Its proof is similar to that of \cite[Proposition 4.2]{KP}, so we omit it.
\begin{proposition}\label{lin prop}
Assume $N \ge 4$, $p \in (1,\frac{N}{N-2})$, $\ep > 0$ small, and $(d,\tau) \in \Lambda_\delta$. Then there exists a constant $C > 0$ independent of $\ep > 0$ and $(d,\tau) \in \Lambda_\delta$ such that
\[\|L_{\ep,(d,\tau)}(\psi,\phi)\|_{X_\ep} \ge C\|(\psi,\phi)\|_{X_\ep} \quad \textup{for all } (\psi,\phi) \in Z_{\ep,(d,\tau)}.\]
\end{proposition}
\noindent Because of the assumption $p > 1$, one can write the operator $L_{\ep,(d,\tau)}$ as the sum of the identity operator on $X_\ep$ and a compact operator; see \cite[Remark 4.1]{KP}.
As a consequence, Proposition \ref{lin prop} and the Fredholm alternative readily imply the following result.
\begin{cor}\label{lin cor}
Under the hypotheses of Proposition \ref{lin prop}, let $(h_1,h_2) \in Z_{\ep,(d,\tau)}$. Then \eqref{lin eq} admits a unique solution $(\psi,\phi) \in Z_{\ep,(d,\tau)}$. Moreover,
\[\|(h_1,h_2)\|_{X_\ep} \ge C\|(\psi,\phi)\|_{X_\ep}\]
where $C > 0$ is the constant in Proposition \ref{lin prop}. In particular, the operator $L_{\ep,(d,\tau)}^{-1}: Z_{\ep,(d,\tau)} \to Z_{\ep,(d,\tau)}$
given as $L_{\ep,(d,\tau)}^{-1}(h_1,h_2) = (\psi,\phi)$ is well-defined and uniformly bounded in $\ep > 0$ and $(d,\tau) \in \Lambda_\delta$.
\end{cor}

\subsection{Nonlinear problems}
Owing to Corollary \ref{lin cor}, the auxiliary equation \eqref{aux eq1} is further reduced to
\begin{equation}\label{aux eq2}
(\psi,\phi) = L_{\ep,(d,\tau)}^{-1}\(\textup{Id}_{X_\ep}-\Pi_{\ep,(d,\tau)}\) \left[ \mci^*_\ep(\mbN_{\ep,(d,\tau)}(\psi,\phi))-\mce_{\ep,(d,\tau)}\right].
\end{equation}
The next proposition concerns its unique solvability.
\begin{proposition}\label{nonlin prop}
Assume $N \ge 4$, $p \in (1,\frac{N-1}{N-2})$, $\ep > 0$ small, and $(d,\tau) \in \Lambda_\delta$.
Then \eqref{aux eq2} has a unique solution $(\psi_{\ep,(d,\tau)},\phi_{\ep,(d,\tau)}) \in Z_{\ep,(d,\tau)}$ satisfying
\begin{equation}\label{nonlin est}
\left\|\(\psi_{\ep,(d,\tau)},\phi_{\ep,(d,\tau)}\)\right\|_{X_\ep} \le C\|\mce_{\ep,(d,\tau)}\|_{X_\ep}
\end{equation}
where $C > 0$ is independent of $\ep > 0$ and $(d,\tau) \in \Lambda_\delta$.
Furthermore, $(\psi_{\ep,(d,\tau)},\phi_{\ep,(d,\tau)}) \in (L^{\infty}(\Omega_\ep))^2$ and the map $(d,\tau) \in \Lambda_\delta \mapsto (\psi_{\ep,(d,\tau)},\phi_{\ep,(d,\tau)}) \in {X_\ep}$ is of $C^1$-class.
\end{proposition}
\begin{proof}
By employing the Banach fixed-point theorem, we deduce the unique existence of the solution $(\psi_{\ep,(d,\tau)},\phi_{\ep,(d,\tau)})$ to \eqref{aux eq2} satisfying \eqref{nonlin est}.
The fact that $(\psi_{\ep,(d,\tau)},\phi_{\ep,(d,\tau)}) \in (L^{\infty}(\Omega_\ep))^2$ will be proved in Appendix \ref{sec bdd}.
The $C^1$-regularity of the map $(d,\tau) \mapsto (\psi_{\ep,(d,\tau)},\phi_{\ep,(d,\tau)})$ is the consequence of the implicit function theorem, the Fredholm alternative and $(\psi_{\ep,(d,\tau)},\phi_{\ep,(d,\tau)}) \in (L^{\infty}(\Omega_\ep))^2$.
\end{proof}

\subsection{Lyapunov-Schmidt reduction}
For $\ep > 0$ small, we set a $C^1$-functional $J_{0\ep}: \Lambda_\delta \to \R$ by
\[J_{0\ep}(d,\tau) = I_\ep\(\mcp U_\mx+\psi_{\ep,(d,\tau)}, PV_\mx+\phi_{\ep,(d,\tau)}\)\]
where $I_\ep$ is the energy functional in \eqref{ene} and $(\psi_{\ep,(d,\tau)},\phi_{\ep,(d,\tau)})$ is the solution to \eqref{aux eq2} described in Proposition \ref{nonlin prop}.
Let $\textup{Int}(\Lambda_\delta)$ be the interior of $\Lambda_\delta$.
\begin{proposition}\label{LSred}
Assume $N \ge 4$, $p \in (1,\frac{N-1}{N-2})$, and $\ep > 0$ small. Suppose that $(d,\tau) \in \textup{Int}(\Lambda_\delta)$ is a critical point of $J_{0\ep}$.
Then $(\mcp U_\mx+\psi_{\ep,(d,\tau)}, PV_\mx+\phi_{\ep,(d,\tau)})$ is a critical point of $I_\ep$ which belongs to $(C^2(\ovom_\ep))^2$, and hence a classical solution to \eqref{main system} having the form \eqref{eq main}.
Moreover, one may assume that its components are positive in $\Omega_\ep$.
\end{proposition}
\begin{proof}
For $(d,\tau) \in \textup{Int}(\Lambda_\delta)$, we set $\tau' = d\tau$ so that $\xi = \mu_\ep \tau'$.
Regarding $d$ and $\tau'$ as independent variables (so that $\tau$ is a dependent variable), we define $\wtj_{0\ep}(d,\tau') = J_{0\ep}(d,\tau)$. It is plain to verify that
\[\nabla J_{0\ep}(d,\tau) = 0 \Leftrightarrow \nabla \wtj_{0\ep}(d,\tau') = 0.\]
By slightly modifying the proof of \cite[Proposition 4.7]{KP}, we see that if $\nabla \wtj_{0\ep}(d,\tau') = 0$,
then $(\mcp U_{\mu_\ep d, \mu_\ep \tau'} +\psi_{\ep,(d,d^{-1}\tau')}, PV_{\mu_\ep d, \mu_\ep \tau'}+\phi_{\ep,(d,d^{-1}\tau')})
= (\mcp U_\mx+\psi_{\ep,(d,\tau)}, PV_\mx+\phi_{\ep,(d,\tau)})$ is a critical point of $I_\ep$.

The rest of the proof also goes along the same lines of the proof of \cite[Proposition 4.7]{KP}.
\end{proof}

The relationship between $J_{0\ep}$ and $J_\ep$ in \eqref{red ene} are given in the following lemma.
\begin{lemma}\label{J0ep Jep}
Assume $N \ge 4$, $p \in (1,\frac{N-1}{N-2})$, and $\ep > 0$ small. Let
\begin{equation}\label{c0c1c2}
c_0 = \frac{2}{N}\int_{\R^N}U^{q+1},\ c_1 = \frac{1}{p+1}\(\frac{b_{N,p}}{\ga_N}\)^p\wth_0(0) \int_{\R^N}U^q > 0,\ c_2 = \ga_N^{-1} > 0,
\end{equation}
and
\begin{equation}\label{theta}
\Theta(d,\tau) = c_1d^{(N-2)p-2} + c_2d^{2-N}U(\tau)V(\tau) \quad \text{for } (d,\tau) \in \Lambda_\delta.
\end{equation}
Then
\begin{equation}\label{J0ep Jep1}
J_{0\ep}(d,\tau) = J_{\ep}(d,\tau) + o(1)\mu_\ep^{(N-2)p-2} = c_0 + \mu_\ep^{(N-2)p-2} \Theta(d,\tau) + o(1)\mu_\ep^{(N-2)p-2}
\end{equation}
and
\begin{equation}\label{J0ep Jep2}
\nabla_{(d,\tau)} J_{0\ep}(d,\tau) = \nabla_{(d,\tau)}J_{\ep}(d,\tau) + o(1)\mu_\ep^{(N-2)p-2} = \mu_\ep^{(N-2)p-2} \nabla_{(d,\tau)} \Theta(d,\tau) + o(1)\mu_\ep^{(N-2)p-2}
\end{equation}
where $C_\ep > 0$ in \eqref{o(1)} is chosen uniformly for $(d,\tau)\in \Lambda_\delta$.
\end{lemma}
\begin{proof}
Let $q^* \in (1,2)$ be the number in \eqref{pq star}. Estimate \eqref{error est} and the inequality that $Nqq^* > (N-2)(q+1)$ imply
\begin{equation}\label{mce q*}
\|\mce_{\ep,(d,\tau)}\|_{X_\ep}^{q^*} = o(1)\mu_\ep^{(N-2)p-2}.
\end{equation}
Employing \eqref{mce q*}, one can check the first inequality in \eqref{J0ep Jep1} as in the proof of \cite[Lemma 5.4]{KP}.

\medskip
We verify the first equality in \eqref{J0ep Jep2}. Let $s$ be one of the parameters $d,\tau_1,\ldots,\tau_N$. We have
\begin{align*}
\nabla_s(J_{0\ep}-J_{\ep})(d,\tau) &= \left[I'_\ep\(\mcp U_\mx+\psi_{\ep,(d,\tau)}, PV_\mx+\phi_{\ep,(d,\tau)}\) - I'_\ep\(\mcp U_\mx, PV_\mx\)\right](\nabla_s\mcp U_\mx,\nabla_sPV_\mx) \\
&\ + I'_\ep\(\mcp U_\mx+\psi_{\ep,(d,\tau)}, PV_\mx+\phi_{\ep,(d,\tau)}\)(\nabla_s\psi_{\ep,(d,\tau)},\nabla_s\phi_{\ep,(d,\tau)})\\
&=: \textup{(VI)} + \textup{(VII)}
\end{align*}

By differentiating \eqref{PUPV eq} and \eqref{mcp U eq} with respect to $s$, we get the equation of $(\nabla_s\mcp U_\mx,\nabla_sPV_\mx)$.
Using it, the third equality in \eqref{ele ineq}, Lemma \ref{nsPUPV est}, H\"older's inequality, the Sobolev inequality, the inequalities that $0 \le \mcp U_\mx \le U_\mx$
and $0 \le PV_\mx \le V_\mx$ in $\Omega_\ep$, \eqref{nonlin est} and \eqref{mce q*}, we compute
\begin{align*}
\textup{(VI)} &= -\int_{\Omega_\ep} \left[|\mcp U_\mx+\psi_{\ep,(d,\tau)}|^{q-1}\(\mcp U_\mx+\psi_{\ep,(d,\tau)}\) - (\mcp U_\mx)^q - q(\mcp U_\mx)^{q-1}\psi_{\ep,(d,\tau)}\right] \nabla_s\mcp U_\mx \\
&\ -\int_{\Omega_\ep} \left[|PV_\mx+\phi_{\ep,(d,\tau)}|^{p-1}\(PV_\mx+\phi_{\ep,(d,\tau)}\) - (PV_\mx)^p - p(PV_\mx)^{p-1}\phi_{\ep,(d,\tau)}\right] \nabla_sPV_\mx \\
&\ -q\int_{\Omega_\ep} \left[(\mcp U_\mx)^{q-1} \nabla_s(\mcp U_\mx-U_\mx)+\nabla_s U_\mx \left\{(\mcp U_\mx)^{q-1}-U_\mx^{q-1}\right\}\right] \psi_{\ep,(d,\tau)} \\
&= O(1) \int_{\Omega_\ep} \left[(\mcp U_\mx)^{q-2} \psi_{\ep,(d,\tau)}^2 + |\psi_{\ep,(d,\tau)}|^q \mone_{q > 2}\right] \left|\nabla_s\mcp U_\mx\right| + (PV_\mx)^{p-2} \phi_{\ep,(d,\tau)}^2 \left|\nabla_sPV_\mx\right| \\
&\ + O(1) \int_{\Omega_\ep} \left[(\mcp U_\mx)^{q-1} \left|\nabla_s(\mcp U_\mx-U_\mx)\right|+U_\mx^{q-1}|\mcp U_\mx-U_\mx|\right] \left|\psi_{\ep,(d,\tau)}\right| \\
&= O(1) \left[\left\|\(\psi_{\ep,(d,\tau)},\phi_{\ep,(d,\tau)}\)\right\|_{X_\ep}^2 \right. \\
&\hspace{35pt} \left. + \(\|\nabla_s(\mcp U_\mx-U_\mx)\|_{L^{q+1}(\Omega_\ep)}+\|\mcp U_\mx-U_\mx\|_{L^{q+1}(\Omega_\ep)}\) \|\psi_{\ep,(d,\tau)}\|_{L^{q+1}(\Omega_\ep)}\right] \\
&= o(1)\mu_\ep^{(N-2)p-2} + O(1) \left[\|\nabla_s(\mcp U_\mx-U_\mx)\|_{L^{q+1}(\Omega_\ep)} + \|\mcp U_\mx-U_\mx\|_{L^{q+1}(\Omega_\ep)}\right] \|\mce_{\ep,(d,\tau)}\|_{X_\ep}.
\end{align*}
Besides, \eqref{mcp U}, \eqref{R C1} and \eqref{mca est C1} yield
\begin{equation}\label{mcp U-U}
\|\nabla_s(\mcp U_\mx-U_\mx)\|_{L^{q+1}(\Omega_\ep)} +\| \mcp U_\mx-U_\mx \|_{L^{q+1}(\Omega_\ep)}= O(1)\left[\mu^{\frac{Np}{q+1}} + \(\frac{\ep}{\mu}\)^{\frac{N}{q+1}}\right].
\end{equation}
By \eqref{simple rel} and \eqref{error est},
\[(\|\nabla_s(\mcp U_\mx-U_\mx)\|_{L^{q+1}(\Omega_\ep)}+\| \mcp U_\mx-U_\mx \|_{L^{q+1}(\Omega_\ep)}) \|\mce_{\ep,(d,\tau)}\|_{X_\ep} = o(1)\mu_\ep^{(N-2)p-2}.\]
Thus it holds that $\textup{(VI)} = o(1)\mu_\ep^{(N-2)p-2}$.

Moreover, by using \eqref{aux eq} and the fact that $(\psi_{\ep,(d,\tau)},\phi_{\ep,(d,\tau)}) \in Z_{\ep,(d,\tau)}$, we find
\begin{equation}\label{IV eq}
\begin{aligned}
\textup{(VII)} &= \sum_{l=0}^N \tc_{\ep,l} \int_{\Omega_\ep} \(p V_\mx^{p-1} \Phi_\mx^l \nabla_s\phi_{\ep,(d,\tau)} + q U_\mx^{q-1} \Psi_\mx^l \nabla_s\psi_{\ep,(d,\tau)}\) \\
&= - \sum_{l=0}^N \tc_{\ep,l} \int_{\Omega_\ep} \left[\nabla_s\(p V_\mx^{p-1} \Phi_\mx^l\) \phi_{\ep,(d,\tau)} + \nabla_s\(q U_\mx^{q-1} \Psi_\mx^l\) \psi_{\ep,(d,\tau)}\right]
\end{aligned}
\end{equation}
where the coefficients $\tc_{\ep,l}$'s are defined by the relation
\begin{multline*}
\sum_{l=0}^N \tc_{\ep,l} \(P\Psi_\mx^l,P\Phi_\mx^l\) = \(\mcp U_\mx+\psi_{\ep,(d,\tau)}, PV_\mx+\phi_{\ep,(d,\tau)}\) \\
- \mci^*_\ep\(|PV_\mx+\phi_{\ep,(d,\tau)}|^{p-1} \(PV_\mx+\phi_{\ep,(d,\tau)}\),|\mcp U_\mx+\psi_{\ep,(d,\tau)}|^{q-1} \(\mcp U_\mx+\psi_{\ep,(d,\tau)}\)\);
\end{multline*}
see \eqref{Pi}. Since
\begin{align*}
p\int_{\Omega_\ep} V_\mx^{p-1} \Phi_\mx^m PV_\mx &= \int_{\Omega_\ep} \nabla P\Psi_\mx^m \cdot \nabla PV_\mx = \int_{\Omega_\ep} P\Psi_\mx^m U_\mx^q &(\textup{by } \eqref{PPsPPhmxl eq} \text{ and } \eqref{PUPV eq}), \\
q\int_{\Omega_\ep} U_\mx^{q-1} \Psi_\mx^m \mcp U_\mx &= \int_{\Omega_\ep} \nabla P\Phi_\mx^m \cdot \nabla \mcp U_\mx = \int_{\Omega_\ep} P\Phi_\mx^m (PV_\mx)^p &(\textup{by } \eqref{PPsPPhmxl eq} \text{ and } \eqref{mcp U eq}),
\end{align*}
if we define a number
\begin{align*}
M_{\ep,(d,\tau)}^m &= \int_{\Omega_\ep} \left[pV_\mx^{p-1} \Phi_\mx^m PV_\mx - |PV_\mx+\phi_{\ep,(d,\tau)}|^{p-1} P\Phi_\mx^m \(PV_\mx+\phi_{\ep,(d,\tau)}\)\right] \\
&\ + \int_{\Omega_\ep} \left[qU_\mx^{q-1} \Psi_\mx^m \mcp U_\mx - |\mcp U_\mx+\psi_{\ep,(d,\tau)}|^{q-1} P\Psi_\mx^m \(\mcp U_\mx+\psi_{\ep,(d,\tau)}\)\right]
\end{align*}
for $m = 0,\ldots,N$, then
\begin{equation}\label{Am1}
\begin{aligned}
M_{\ep,(d,\tau)}^m &= - \int_{\Omega_\ep} \left[|PV_\mx+\phi_{\ep,(d,\tau)}|^{p-1} \(PV_\mx+\phi_{\ep,(d,\tau)}\)-(PV_\mx)^p\right] P\Phi_\mx^m \\
&\ - \int_{\Omega_\ep} \left[|\mcp U_\mx+\psi_{\ep,(d,\tau)}|^{q-1} \(\mcp U_\mx+\psi_{\ep,(d,\tau)}\)-(\mcp U_\mx)^q\right] P\Psi_\mx^m \\
&\ - \int_{\Omega_\ep} \left[(\mcp U_\mx)^q-U_\mx^q\right] P\Psi_\mx^m.
\end{aligned}
\end{equation}
By \eqref{Am1}, \eqref{error est}, \eqref{nonlin est}, \eqref{mcp U-U}, the inequalities that $|P\Psi_\mx^m| \le C\mu^{-1} U_\mx$ and $|P\Phi_\mx^m| \le C\mu^{-1} V_\mx$ in $\Omega_\ep$, it follows that
\begin{equation}\label{Am2}
\begin{aligned}
M_{\ep,(d,\tau)}^m &= O(1) \left\|P\Phi_\mx^m\right\|_{L^{p+1}(\Omega_\ep)} \left\|\phi_{\ep,(d,\tau)}\right\|_{L^{p+1}(\Omega_\ep)} \\
&\ + O(1) \left\|P\Psi_\mx^m\right\|_{L^{q+1}(\Omega_\ep)} \(\left\|\psi_{\ep,(d,\tau)}\right\|_{L^{q+1}(\Omega_\ep)} + \left\|\mcp U_\mx-U_\mx\right\|_{L^{q+1}(\Omega_\ep)}\) \\
&= O(1)\mu^{-1}\left[\mu^{\frac{Np}{q+1}} + \(\frac{\ep}{\mu}\)^{\frac{N}{q+1}}\right].
\end{aligned}
\end{equation}
Testing \eqref{aux eq} with $(P\Psi_\mx^m,P\Phi_\mx^m)$ for each $m = 0,\ldots,N$ and applying $(\psi_{\ep,(d,\tau)},\phi_{\ep,(d,\tau)}) \in Z_{\ep,(d,\tau)}$ and \eqref{Am2}, we deduce
\begin{equation}\label{tc epl}
\begin{aligned}
\tc_{\ep,m} &= \mu^2 \sum_{l=0}^N (\mone_{m=l} + o(1)) M_{\ep,(d,\tau)}^m \left[\int_{\R^N} \(pV^{p-1}(\Phi_{1,0}^m)^2 + qU^{q-1}(\Psi_{1,0}^m)^2\)\right]^{-1}\\
&= O(1)\mu\left[\mu^{\frac{Np}{q+1}} + \(\frac{\ep}{\mu}\)^{\frac{N}{q+1}}\right].
\end{aligned}
\end{equation}
By combining \eqref{IV eq}, \eqref{tc epl}, \eqref{error est} and
\begin{multline*}
\int_{\Omega_\ep} \left[\nabla_s\(p V_\mx^{p-1} \Phi_\mx^l\) \phi_{\ep,(d,\tau)} + \nabla_s\(q U_\mx^{q-1} \Psi_\mx^l\) \psi_{\ep,(d,\tau)}\right]\dx\\
= O(1) \(\big\|\Phi_\mx^l\big\|_{L^{p+1}(\Omega_\ep)} \|\phi_{\ep,(d,\tau)}\|_{L^{p+1}(\Omega_\ep)}
+ \big\| \Psi_\mx^l\big\|_{L^{q+1}(\Omega_\ep)} \|\psi_{\ep,(d,\tau)}\|_{L^{q+1}(\Omega_\ep)}\) = O(1)\mu^{-1} \|\mce_{\ep,(d,\tau)}\|_{X_\ep},
\end{multline*}
we arrive at $\textup{(VII)} = o(1)\mu_\ep^{(N-2)p-2}$.

Consequently, the first equality in \eqref{J0ep Jep2} is true.

\medskip
The second equalities in \eqref{J0ep Jep1} and \eqref{J0ep Jep2} follow from \eqref{ene exp} and \eqref{ene exp C1}, respectively.
\end{proof}

\section{Completion of the proof of Theorem \ref{thm main}}\label{sec comp}
We are now ready to conclude the proof of our main theorem.
Below, we show the function $\Theta$ in \eqref{theta} has a non-degenerate saddle point $(\td,\tta) \in \textup{Int}(\Lambda_\delta)$ with a suitable choice of $\delta > 0$.
Then Lemma \ref{J0ep Jep} and a simple degree argument guarantee the existence of a critical point of $J_{0\ep}$ in $\text{Int}(\Lambda_\delta)$ for $\ep > 0$ small.
Proposition \ref{LSred} implies the existence of a small number $\ep_0 > 0$ and a family of solutions $\{(u_\ep,v_\ep)\}_{\ep \in (0,\ep_0)}$ to system \eqref{main system} depicted in the statement of Theorem \ref{thm main}.

Let
\[\td = \left[\frac{c_2(N-2)V(0)}{c_1((N-2)p-2)}\right]^{\frac{1}{(N-2)p+N-4}} > 0 \quad \text{and} \quad \tta = 0 \in \R^N\]
where $c_1,\, c_2 > 0$ are the numbers in \eqref{c0c1c2}.
Then, using $U(0) = 1$ and $\nabla U(0) = \nabla V(0) = 0$, we easily check that $\nabla_{(d,\tau)} \Theta(\td,\tta) = 0$. Furthermore, it holds that
\[\pa^2_{d\tau_l}\Theta(\td,\tta) = 0 \quad \text{for each } l = 1,\ldots,N,\]
\begin{align*}
\pa^2_{dd}\Theta(\td,\tta) &= \td^{-N}\left[c_1((N-2)p-2)((N-2)p-3)\td^{(N-2)p+N-4} + c_2(N-2)(N-1)(UV)(0)\right] \\
&= \td^{-N} c_2(N-2)((N-2)p+N-4)V(0) > 0,
\end{align*}
and
\begin{align*}
\pa^2_{\tau_l\tau_m}\Theta(\td,\tta) &= c_2\td^{2-N} \left[(\pa^2_{x_lx_m}U)(0)V(0) + (\pa^2_{x_lx_m}V)(0)U(0)\right] \\
&= \begin{cases}
c_2\td^{2-N} \left[(\pa^2_{rr}U)(0)V(0) + (\pa^2_{rr}V)(0)\right] = - c_2N^{-1}\td^{2-N}\(V^{p+1}(0)+1\) < 0 &\text{if } l = m,\, \footnotemark{} \\
0 &\text{if } l \ne m
\end{cases}
\end{align*}%
\footnotetext{Let $r = |x|$. Abusing notation, we write $U(r) = U(x)$ and $V(r) = U(x)$. Then system \eqref{main system} is rewritten as
\[-(\pa^2_{rr}U)(r) - \frac{N-1}{r}(\pa_rU)(r) = V^p(r) \quad \text{and} \quad -(\pa^2_{rr}V)(r) - \frac{N-1}{r}(\pa_rV)(r) = U^q(r) \quad \text{for } r \in (0,\infty).\]
Taking $r \to 0$ on the both equations, we find $-N(\pa^2_{rr}U)(0) = V^p(0)$ and $-N(\pa^2_{rr}V)(0) = U^q(0) = 1$.}for
$l, m = 1,\ldots,N$. Consequently, the Hessian matrix $D^2_{(d,\tau)}\Theta(\td,\tta)$ of $\Theta$ at $(\td,\tta)$ has one positive eigenvalue and $N$ negative eigenvalues.
Taking $\delta = \td+\td^{-1}$, we see that $(\td,\tta)$ is a non-degenerate saddle point of $\Theta$ in $\textup{Int}(\Lambda_\delta)$. The proof of Theorem \ref{thm main} is now finished.

\appendix
\section{Technical computations}\label{sec tech}
An elementary calculus yields the following inequalities.
\begin{lemma}
Assume that $1 < p < 2$. Then there exists a constant $C > 0$ depending only on $p$ such that
\begin{equation}\label{ele ineq}
\begin{cases}
\left||a+b|^{p-1}(a+b)-|a|^{p-1}a\right| \le C\(|a|^{p-1}|b|+|b|^p\),\\
\left||a+b|^{p-1}(a+b)-|a|^{p-1}a-|b|^{p-1}b\right| \le C|a|^{p-1}|b|,\\
\left||a+b|^{p-1}(a+b) - |a|^{p-1}a - p|a|^{p-1}b\right| \le C\min\left\{|a|^{p-2}b^2,|b|^p\right\},\\
|a|\left||a+b|^{p-2}(a+b) - |a|^{p-2}a - (p-1)|a|^{p-2}b\right| \le C|b|^p
\end{cases}
\end{equation}
for any $a, b \in \R$.
\end{lemma}

\begin{lemma}
Let $N \ge 3$, $a < 2$, $\lambda \in (0,1)$ and $\xi \in \R^N$. Then there exists a constant $C > 0$ depending only on $N$ and $a$ such that
\begin{equation}\label{interior es}
\int_{B(\xi,\lambda)} \frac{\dy}{|x-y|^{N-2}|y-\xi|^a} \le \frac{C\lambda^{2-a}}{1+|\lambda^{-1}(x-\xi)|^{N-2}} \quad \textup{for } x \in \R^N.
\end{equation}
\end{lemma}
\begin{proof}
Applying a change of variable $y-\xi \mapsto \lambda y$, we observe
\[\int_{B(\xi,\lambda)}\frac{\dy}{|x-y|^{N-2}|y-\xi|^a} = \lambda^{2-a} \int_{B(0,1)}\frac{\dy}{|z-y|^{N-2}|y|^a}\]
where $z := \lambda^{-1}(x-\xi)$.

Since
$$
\int_{B(0,1)}\frac{\dy}{|z-y|^{N-2}|y|^a} \le C|z|^{2-N}\int_{B(0,1)} |y|^{-a}\, \dy \le C|z|^{2-N} \quad \textup{for } |z| \ge 2,
$$
we have
$$
\int_{B(\xi,\lambda)}\frac{\dy}{|x-y|^{N-2}|y-\xi|^a} \le C\lambda^{2-a} \left|\lambda^{-1}(x-\xi)\right|^{2-N} \quad \textup{provided } |x-\xi| \ge 2\lambda.
$$
Besides, it holds that
\[-\Delta_z \int_{B(0,2)}\frac{\dy}{|z-y|^{N-2}|y|^a} = \ga_N^{-1} |z|^{-a}\in L^q(B(0,2)) \quad \textup{for some } q > \frac{N}{2}.\]
By elliptic regularity, we conclude that
\[\int_{B(\xi,\lambda)}\frac{\dy}{|x-y|^{N-2}|y-\xi|^a} \le C\lambda^{2-a} \quad \textup{provided } |x-\xi| < 2\lambda. \qedhere\]
\end{proof}

In the rest of the appendix, we assume that $(d,\tau) \in \Lambda_\delta$ (see \eqref{Lambda delta}) and \eqref{mu ep}--\eqref{trans and scaling} holds.
\begin{lemma}\label{basic est 2}
Let $N \ge 3$, $p \in (1,\frac{N}{N-2})$, $a>0$ and $b \in \mathbb{R}$. Then we have
\begin{equation}\label{basic est 21}
\int_{\Omega_\ep} U_\mx^a(x) |x|^{-b} \dx = \begin{cases}
O(1)\mu^{-\frac{Na}{q+1}+N-b}\(\dfrac{\ep}{\mu}\)^{N-b} &\textup{for } b>N,\\
O(1)\mu^{-\frac{Na}{q+1}} \left|\ln\(\dfrac{\ep}{\mu}\)\right| &\textup{for } b=N.
\end{cases}
\end{equation}
If we further assume that $b < N$, then
\begin{equation}\label{basic est 22}
\int_{\Omega_\ep} U_\mx^a(x) |x|^{-b}\, \dx=\begin{cases}
O(1) \mu^{-\frac{Na}{q+1}+N-b} &\textup{for } ((N-2)p-2)a+b>N,\\
O(1) \mu^{-\frac{Na}{q+1}+N-b}|\ln \mu|&\textup{for } ((N-2)p-2)a+b=N,\\
O(1) \mu^{-\frac{Na}{q+1}+((N-2)p-2)a} &\textup{for } ((N-2)p-2)a+b<N.
\end{cases}
\end{equation}
\end{lemma}
\begin{proof}
By \eqref{decay for u}, we have
\[\int_{\Omega_\ep} U_\mx^a(x) |x|^{-b} \dx 
=O(1)\mu^{-\frac{Na}{q+1}+N-b} \int_{\mu^{-1}\Omega_\ep} \frac{\dx}{(1+|x|^{((N-2)p-2)a})|x|^b}.\]
If $b\ge N$, then $((N-2)p-2)a+b>N$, and hence
$$
\int_{\mu^{-1}\Omega_\ep}\frac{\dx}{(1+|x|^{((N-2)p-2)a})|x|^b} = \begin{cases}
O(1) \(\dfrac{\ep}{\mu}\)^{N-b} &\textup{for } b>N,\\
O(1) \left|\ln\(\dfrac{\ep}{\mu}\)\right| &\textup{for } b=N.
\end{cases}
$$
On the other hand, if $ b<N$, then
$$
\int_{\mu^{-1}\Omega_\ep}\frac{\dx}{(1+|x|^{((N-2)p-2)a})|x|^b} = \begin{cases}
O(1) &\textup{for } ((N-2)p-2)a+b>N,\\
O(1) |\ln \mu| &\textup{for } ((N-2)p-2)a+b=N,\\
O(1) \mu^{-N+((N-2)p-2)a+b} &\textup{for } ((N-2)p-2)a+b<N.
\end{cases}
$$
Combining all the computations, we derive \eqref{basic est 21} and \eqref{basic est 22}.
\end{proof}

\begin{lemma}
Let $N \ge 3$ and $\kappa \in (0,1)$ small. It holds that
\begin{equation}\label{basic est whole}
\int_{\R^N \setminus B(0,1)} \frac{\dy}{|x-y|^{N-2}|y|^a} =\begin{cases}
O(1)\min\left\{1,|x|^{2-a}\right\} &\textup{if } 2<a<N,\\
O(1)\min\left\{1,|x|^{2-N}|\ln |x||\right\} &\textup{if } a=N,\\
O(1)\min\left\{1,|x|^{2-N}\right\} &\textup{if } a>N
\end{cases}
\end{equation}
for $x \in \R^N$. In addition,
\begin{align}
\int_{\Omega \setminus B(0,\mu^\kappa)} \frac{\dy}{|x-y|^{N-2}|y|^b} &= \begin{cases}
O(1) &\textup{if } b<2,\\
O(1) |\ln\mu| &\textup{if } b=2,\\
O(1)\mu^{\kappa(2-b)} &\textup{if } 2<b<N,
\end{cases}\label{basic est}\\
\int_{\Omega} \frac{\dy}{|x-y|^{N-2}|y|^b} &= \begin{cases}
O(1) &\textup{if } b<2,\\
O(1) |\ln|x|| &\textup{if } b=2,\\
O(1)|x|^{2-b} &\textup{if } 2<b<N
\end{cases}\label{basic est o}
\end{align}
for $x \in \Omega$.
\end{lemma}
\begin{proof}
We will only verify \eqref{basic est}. Estimate \eqref{basic est whole} is well-known, and the proof of \eqref{basic est} essentially covers that of \eqref{basic est o}. We write $\Omega_{\mu^\kappa} = \Omega \setminus B(0,\mu^\kappa)$.

If $|x|\le \frac{\mu^\kappa}{2}$, then $|x-y|\ge \frac{|y|}{2}$ for $y \in \Omega_{\mu^\kappa}$. Thus
\[\int_{\Omega_{\mu^\kappa}} \frac{\dy}{|x-y|^{N-2}|y|^b}\le C\int_{\Omega_{\mu^\kappa}} \frac{\dy}{|y|^{N+b-2}} \le \begin{cases}
O(1) &\textup{if } b<2, \\
O(1) |\ln\mu| &\textup{if } b=2, \\
O(1)\mu^{\kappa(2-b)} &\textup{if } b>2.
\end{cases}\]
We next suppose that $|x|> \frac{\mu^\kappa}{2}$. By applying the inequalities
\[\begin{cases}
|y| \ge |x|-|x-y| > \dfrac{|x|}{2} &\textup{for } y \in B\big(x,\frac{|x|}{2}\big),\\
|x-y| \ge |x|-|y| > \dfrac{|x|}{2} &\textup{for } y \in B\big(0,\frac{|x|}{2}\big),\\
|y| \le |y-x|+|x| \le 3|y-x| &\textup{for } y \in \R^N \setminus \left[B\big(x,\frac{|x|}{2}\big) \cup B\big(0,\frac{|x|}{2}\big)\right],
\end{cases}\]
we derive
\begin{align*}
\int_{\Omega_{\mu^\kappa} \cap B(x,\frac{|x|}{2})} \frac{\dy}{|x-y|^{N-2}|y|^b} &= O(1)|x|^{-b} \int_{B(x,\frac{|x|}{2})}\frac{\dy}{|x-y|^{N-2}} = O(1)|x|^{2-b},\\
\int_{\Omega_{\mu^\kappa} \cap B(0,\frac{|x|}{2})} \frac{\dy}{|x-y|^{N-2}|y|^b} &= O(1)|x|^{2-N} \int_{B(0,\frac{|x|}{2})} \frac{\dy}{|y|^b} = O(1)|x|^{2-b} \quad (\textup{since } b<N),\\
\int_{\Omega_{\mu^\kappa} \setminus (B(x,\frac{|x|}{2}) \cup B(0,\frac{|x|}{2}))} \frac{\dy}{|x-y|^{N-2}|y|^b} &= O(1) \int_{\Omega \setminus B(0,\frac{|x|}{2})} \frac{\dy}{|y|^{N-2+b}} = \begin{cases}
O(1) &\textup{if } b<2,\\
O(1) |\ln|x|| &\textup{if } b=2,\\
O(1)|x|^{2-b} &\textup{if } b>2.
\end{cases}
\end{align*}
This completes the verification of \eqref{basic est}.
\end{proof}

\begin{lemma}\label{high est}
Let $N \ge 4$, $p \in (1,\frac{N}{N-2})$ and $\kappa\in (0,\frac{1}{N-1})$. Let $R_3$ be the function in \eqref{R3}. It holds that
\begin{equation}\label{R3 exp}
R_3(x) = O(1)\mu^{\frac{Np}{q+1}} \left[\mu^{1+\kappa(1-N)} + \(\frac{\ep}{\mu}\)^{N-2}\mu^{\kappa(2-(N-2)p)} + \ep^{N-2}\mu^{\kappa(2-(N-2)p)}\right]
\end{equation}
for $x \in \Omega_\ep$.
\end{lemma}
\begin{proof}
Let $\Omega_{\mu^\kappa} = \Omega_\ep \setminus B(0,\mu^\kappa)$. A direct computation shows that
\begin{align*}
R_3(x) = O(1)\int_{\Omega_{\mu^\kappa}} \frac{1}{|x-y|^{N-2}} &\left[\frac{\mu^{\frac{Np}{q+1}+2-N}\ep^{N-2}}{|y|^{(N-2)p}} + \big|\wtr_2(y)\big|^p \right. \\
&\ \left. + \left\{1+\(\frac{\ep}{\mu}\)^{(N-2)(p-1)}\right\} \frac{\mu^{\frac{N(p-1)}{q+1}}}{|y|^{(N-2)(p-1)}} \big|\wtr_2(y)\big| \right]\dy.
\end{align*}
By \eqref{basic est} and \eqref{wtr 2}, we discover
\[\mu^{\frac{Np}{q+1}} \int_{\Omega_{\mu^\kappa}} \frac{\mu^{2-N} \ep^{N-2}}{|x-y|^{N-2}|y|^{(N-2)p}}\, \dy
= O(1)\mu^{\frac{Np}{q+1}} \(\frac{\ep}{\mu}\)^{N-2}\mu^{\kappa(2-(N-2)p)},\]
\begin{multline*}
\int_{\Omega_{\mu^\kappa}} \frac{1}{|x-y|^{N-2}}\big|\wtr_2(y)\big|^p \dy\\
=O(1) \mu^{\frac{Np}{q+1}} \left[\mu^{(1+\kappa(1-N))p}+\(\frac{\ep}{\mu}\)^{(N-2)p} + \ep^{(N-2)p}\mu^{\kappa(2-(N-2)p)} + \(\frac{\ep}{\mu}\)^{(N-1)p}\mu^{\kappa(2-(N-2)p)}\right],
\end{multline*}
and
\begin{multline*}
\mu^{\frac{N(p-1)}{q+1}} \int_{\Omega_{\mu^\kappa}} \frac{1}{|x-y|^{N-2}|y|^{(N-2)(p-1)}} \ \big|\wtr_2(y)\big|\dy\\
=O(1)\mu^{\frac{Np}{q+1}} \left[\mu^{1+\kappa(1-N)} + \(\frac{\ep}{\mu}\)^{N-2} + \ep^{N-2}\mu^{\kappa(2-(N-2)p)} + \(\frac{\ep}{\mu}\)^{N-1}\mu^{\kappa(2-(N-2)p)}\right].
\end{multline*}
Combining these, we deduce \eqref{R3 exp}.
\end{proof}

\begin{lemma}\label{high est2}
Let $N \ge 4$ and $p \in (1,\frac{N}{N-2})$. If $Q_\mx$ is the integral in \eqref{Q mu}, then
\begin{equation}\label{Q mu est}
Q_\mx = o(1)\mu^{(N-2)p-2}.
\end{equation}
\end{lemma}
\begin{proof}
A straightforward computation using \eqref{decay for u} shows that
\[\mu^{\frac{2Np}{q+1}}\int_{\Omega_\ep}U_\mx^{q-1}
=o(1)\mu^{(N-2)p-2}.\]
Also, in light of Lemma \ref{basic est 2} and \eqref{mca est}, we have
\[\(\ep^{N-2}\mu^{-\frac{N}{q+1}}\)^2\int_{\Omega_\ep}U_\mx^{q-1} {|x|^{2(2-N)}}\, \dx
= O(1)\(\frac{\ep}{\mu}\)^N\(1 + \left|\ln \frac{\ep}{\mu}\right|\mone_{N=4}\),\]
\[\(\ep^{(N-2)p}\mu^{-\frac{Np}{p+1}}\)^2 \int_{\Omega_\ep} U_\mx^{q-1} |x|^{2(2-(N-2)p)} \dx
= o(1)\(\frac{\ep}{\mu}\)^{N-2}\]
and
\[\(\ep^{N-2}\mu^{-\frac{N}{q+1}}\)^2 \int_{\Omega_\ep} U_\mx^{q-1}\mca_{\ep,(d,\tau)}^2
= o(1)\(\frac{\ep}{\mu}\)^{N-2}.\]
Collecting all the estimates and appealing \eqref{simple rel}, we obtain \eqref{Q mu est}.
\end{proof}

\begin{lemma}\label{reg part ene est}
Let $N \ge 4$, $\kappa\in (0,1)$ and $p \in (1,\frac{N}{N-2})$. Then it holds that
\begin{equation}\label{reg part ene est eq}
\ep^{N-2} \int_{\mu^{-1}(\Omega_\ep-\xi)} \int_{B(0,\mu^{\kappa-1})\setminus B(-\tau,\frac{\ep}{\mu})} \frac{H_\ep(\mu x+\xi,\mu y+\xi)U^q(x) V^{p-1}(y)}{|y+\tau|^{N-2}} \dy\dx
= o(1) \(\frac{\ep}{\mu}\)^{N-2}.
\end{equation}
\end{lemma}
\begin{proof}
An application of Fubini's theorem, \eqref{H ep1} and \eqref{decay for v} yields
\begin{align*}
&\ \ep^{N-2} \int_{\mu^{-1}(\Omega_\ep-\xi)} \int_{B(0,\mu^{\kappa-1})\setminus B(-\tau,\frac{\ep}{\mu})} \frac{H (\mu x+\xi,\mu y+\xi) U^q(x) V^{p-1}(y)}{|y+\tau|^{N-2}}\, \dy\dx \\
&= O(1)\ep^{N-2} \int_{\R^N} U^q(x)\dx \int_{B(0,\mu^{\kappa-1})\setminus B(-\tau,\frac{\ep}{\mu})} \frac{V^{p-1}(y)}{|y+\tau|^{N-2}}\, \dy \\
&= O(1)\ep^{N-2} \int_{B(0,\mu^{-1})} \frac{1}{1+|y|^{(N-2)(p-1)}}\frac{\dy}{|y|^{N-2}} \\ 
&= O(1)\(\frac{\ep}{\mu}\)^{N-2}\mu^{(N-2)p-2} = o(1)\(\frac{\ep}{\mu}\)^{N-2},
\end{align*}
\begin{align*}
&\ \ep^{N-2} \int_{\mu^{-1}(\Omega_\ep-\xi)} \int_{B(0,\mu^{\kappa-1})\setminus B(-\tau,\frac{\ep}{\mu})} \frac{H_{\ep,1}(\mu x+\xi,\mu y+\xi)U^q(x) V^{p-1}(y)}{|y+\tau|^{N-2}}\, \dy\dx \\
&= O(1)\(\frac{\ep}{\mu}\)^{2(N-2)} \int_{B(0,\mu^{\kappa-1})\setminus B(-\tau,\frac{\ep}{\mu})}\int_{\R^N}\Big|x+\tau-\(\frac{\ep}{\mu}\)^2\frac{y+\tau}{|y+\tau|^2}\Big|^{2-N}U^q(x)\dx \frac{V^{p-1}(y)}{|y+\tau|^{2(N-2)}}\, \dy \\
&= O(1)\(\frac{\ep}{\mu}\)^{2(N-2)} \int_{B(0,\mu^{\kappa-1})\setminus B(-\tau,\frac{\ep}{\mu})} V\(\tau-\(\frac{\ep}{\mu}\)^2\frac{y+\tau}{|y+\tau|^2}\) \frac{V^{p-1}(y)}{|y+\tau|^{2(N-2)}}\, \dy \\
&= O(1)\(\frac{\ep}{\mu}\)^{2(N-2)} \int_{\R^N\setminus B(0,\frac{\ep}{\mu})} \frac{1}{1+|y|^{(N-2)(p-1)}} \frac{\dy}{|y|^{2(N-2)}}
= o(1)\(\frac{\ep}{\mu}\)^{N-2}
\end{align*}
and
\begin{align*}
&\ \frac{\ep^{2(N-2)}}{\mu^{N-2}} \int_{\mu^{-1}(\Omega_\ep-\xi)} \int_{B(0,\mu^{\kappa-1})\setminus B(-\tau,\frac{\ep}{\mu})} \(| x+\tau|^{2-N}+| y+\tau|^{2-N}\)\frac{U^q(x) V^{p-1}(y)}{|y+\tau|^{N-2}}\, \dy\dx \\
&=O(1)\frac{\ep^{2(N-2)}}{\mu^{N-2}} \int_{B(0,\mu^{\kappa-1})\setminus B(-\tau,\frac{\ep}{\mu})} \left[\int_{\R^N} \frac{U^q(x)}{|x+\tau|^{N-2}} \dx\frac{V^{p-1}(y)}{|y+\tau|^{N-2}} + \frac{V^{p-1}(y)}{|y+\tau|^{2(N-2)}}\right]\dy \\
&=O(1)\frac{\ep^{2(N-2)}}{\mu^{N-2}} \int_{B(\tau,\mu^{\kappa-1})\setminus B(0,\frac{\ep}{\mu})} \(\frac{1}{|y|^{N-2}}+ \frac{1}{|y|^{2(N-2)}}\) \frac{\dy}{1+|y|^{(N-2)(p-1)}}
= o(1) \(\frac{\ep}{\mu}\)^{N-2}.
\end{align*}
Appealing the above computations and Lemma \ref{reg est for puntured}, we establish \eqref{reg part ene est eq}.
\end{proof}

\begin{lemma}\label{nsPUPV est}
Given $(d,\tau) \in \Lambda_\delta$, let $s$ be one of the parameters $d,\tau_1,\ldots,\tau_N$. Then there exists a constant $C > 0$ depending only on $N$, $p$ and $\delta$ such that
\begin{equation}\label{nsPUPV est0}
|\nabla_s\mcp U_\mx(x)| \le C\mcp U_\mx(x) \quad \text{and} \quad |\nabla_s PV_\mx(x)| \le CPV_\mx(x) \quad \text{for } x \in \Omega_\ep.
\end{equation}
\end{lemma}
\begin{proof}
In light of \eqref{UVmx}, \eqref{PUPV eq} and Lemma \ref{limit decay}, we know
\[\begin{cases}
-\Delta \(C PV_\mx \pm \nabla_s PV_\mx\) = U_\mx^{q-1} \(CU_\mx \pm q\nabla_sU_\mx\) \ge 0 &\textup{in } \Omega_\ep,\\
C PV_\mx \pm \nabla_s PV_\mx=0 &\textup{on } \pa \Omega_\ep.
\end{cases}\]
for some $C > 0$ depending only on $N$, $p$ and $\delta$. The maximum principle shows that the second inequality in \eqref{nsPUPV est0} holds.
Also, applying an analogous argument to \eqref{mcp U eq}, we see that the first inequality in \eqref{nsPUPV est0} is true.
\end{proof}

\section{Proof of $(\psi_{\ep,(d,\tau)},\phi_{\ep,(d,\tau)}) \in (L^{\infty}(\Omega_\ep))^2$}\label{sec bdd}
Fix $\ep > 0$ small and $(d,\tau) \in \Lambda_\delta$. In this appendix, we prove that $(\psi_{\ep,(d,\tau)},\phi_{\ep,(d,\tau)}) \in (L^{\infty}(\Omega_\ep))^2$ as stated in Proposition \ref{nonlin prop}.

\medskip
Equation \eqref{aux eq2} reads
\begin{equation}\label{eq:1}
\begin{cases}
\displaystyle -\Delta \psi = \left|\phi + PV_\mx\right|^{p-1} \(\phi + PV_\mx\) + Q_{11} &\textup{in } \Omega_\ep,\\
\displaystyle -\Delta \phi = \left|\psi + \mcp U_\mx\right|^{q-1} \(\psi + \mcp U_\mx\) + Q_{21} &\textup{in } \Omega_\ep,\\
\psi = \phi = 0 &\textup{on } \pa \Omega_\ep
\end{cases}
\end{equation}
where $Q_{11}, Q_{21} \in C^{\infty}(\overline{\Omega_\ep})$.

From the third inequality in \eqref{ele ineq}, we find
\[\left|\phi + PV_\mx\right|^{p-1} \(\phi + PV_\mx\) = |\phi|^{p-1} \phi + p\, |\phi|^{p-1} PV_\mx + O\(PV_\mx^p\) = |\phi|^{p-1} \(\phi +P_1\) + Q_{12}\]
where $P_1 := p\, PV_\mx \in C^{\infty}(\overline{\Omega_\ep})$ and $Q_{12} \in L^{\infty}(\Omega_\ep)$. On the other hand, $q$ may be greater than 2, so we have
\begin{align*}
\left|\psi + \mcp U_\mx\right|^{q-1} \(\psi + \mcp U_\mx\)
&= |\psi|^{q-1} \psi + q |\psi|^{q-1} \mcp U_\mx + O\(|\psi|^{q-2}(\mcp U_\mx)^2 \mone_{q>2} + (\mcp U_\mx)^q\) \\
&= |\psi|^{q-1} (\psi + P_2) + O\(|\psi|^{q-2} \mone_{q>2}\) + Q_{22}
\end{align*}
where $P_2 := q\, \mcp U_\mx \in C^{\infty}(\overline{\Omega_\ep})$ and $Q_{22} \in L^{\infty}(\Omega_\ep)$.

Consequently, \eqref{eq:1} is reduced to
\begin{equation}\label{eq:2}
\begin{cases}
\displaystyle -\Delta \psi = |\phi|^{p-1} (\phi + P_1) + Q_{13} &\textup{in } \Omega_\ep,\\
\displaystyle -\Delta \phi = |\psi|^{q-1} (\psi + P_2) + O\(|\psi|^{q-2} \mone_{q>2}\) + Q_{23} &\textup{in } \Omega_\ep,\\
\psi = \phi = 0 &\textup{on } \pa \Omega_\ep
\end{cases}
\end{equation}
where $(P_1, P_2),\, (Q_{13},Q_{23}) \in (L^{\infty}(\Omega_\ep))^2$.

At this moment, we need a regularity result which holds for any linear Hamiltonian-type elliptic system \eqref{eq:reg1}.
It is a modified version of \cite[Lemma B.1]{KP} and its proof is inspired by that of \cite[Theorem 1.3]{CL}.
\begin{lemma}\label{lemma:reg}
Suppose that $N \ge 4$, $p \in (1,\frac{N-1}{N-2})$, and $\oms$ is a smooth bounded domain in $\R^N$.
Given sufficiently small numbers $\zeta_1,\, \zeta_2 > 0$ satisfying $\zeta_1(q+1) = \zeta_2(p+1)$, we set
\[\sigma_1 = \frac{p+1}{p-\zeta_1} \quad \textup{and} \quad \sigma_2 = \frac{q+1}{q-\zeta_2}.\]
Suppose that $(\psi,\phi) \in X_{p,q}(\oms)$ and $(Q_1,Q_2) \in L^{\sigma_1}(\oms) \times L^{\sigma_2}(\oms)$.
There is a small constant $\delta > 0$ depending only on $N,\, p,\, \oms,\, \zeta_1,\, \zeta_2$ such that if
\begin{equation}\label{eq:reg0}
\|F_1\|_{L^{\frac{p+1}{p-1}}(\oms)} + \|F_2\|_{L^{\frac{q+1}{q-1}}(\oms)} < \delta
\end{equation}
and
\begin{equation}\label{eq:reg1}
\begin{cases}
-\Delta \psi = F_1\phi + Q_1 &\textup{in } \oms,\\
-\Delta \phi = F_2\psi + Q_2 &\textup{in } \oms,\\
\psi = \phi = 0 &\textup{on } \pa \oms,
\end{cases}
\end{equation}
then $(\psi, \phi) \in L^{\frac{N\sigma_1}{N-2\sigma_1}}(\oms) \times L^{\frac{N\sigma_2}{N-2\sigma_2}}(\oms)$.
\end{lemma}
\begin{proof}
Let $(r,s) = (\frac{N\sigma_1}{N-2\sigma_1}, \frac{N\sigma_2}{N-2\sigma_2})$. Then
\begin{equation}\label{eq:reg11}
r > \frac{N(p+1)/p}{N-2(p+1)/p} = q+1,\quad s > \frac{N(q+1)/q}{N-2(q+1)/q} = p+1 > \frac{N}{N-2},
\end{equation}
and
\begin{equation}\label{eq:reg12}
\begin{aligned}
\frac{1}{r} + \frac{1}{p+1} - \frac{1}{q+1} &= \frac{1}{\sigma_1} - \frac{2}{N} + \frac{1}{p+1} - \frac{1}{q+1} \\
&= \frac{p}{p+1} - \frac{2}{N} + \frac{1}{p+1} - \frac{1}{q+1} - \frac{\zeta_1}{p+1} = \frac{1-\zeta_1}{p+1} \\
&= \frac{1}{p+1} - \frac{\zeta_2}{q+1} = \frac{q-\zeta_2}{q+1} - \frac{2}{N} = \frac{1}{\sigma_2} - \frac{2}{N} = \frac{1}{s}.
\end{aligned}
\end{equation}
Let $T_1$ and $T_2$ be the operators given as
\[(T_1g)(x) = \int_{\oms} G(x,y) (F_1g)(y)\dy \quad \textup{and} \quad (T_2f)(x) = \int_{\oms} G(x,y) (F_2f)(y)\dy\]
for $x \in \oms$, where $G$ is the Green's function of the Dirichlet Laplacian $-\Delta$ in $\oms$.
Applying the Hardy-Littlewood-Sobolev inequality, H\"older's inequality, \eqref{eq:reg11} and \eqref{eq:reg12}, we obtain
\[\|T_1g\|_{L^r(\oms)} \le C\|F_1g\|_{L^{\frac{Nr}{N+2r}}(\oms)} \le C\|F_1\|_{L^{\frac{p+1}{p-1}}(\oms)}\|g\|_{L^s(\oms)},\]
and similarly,
\[\|T_2f\|_{L^s(\oms)} \le C\|F_2f\|_{L^{\frac{Ns}{N+2s}}(\oms)} \le C\|F_2\|_{L^{\frac{q+1}{q-1}}(\oms)}\|f\|_{L^r(\oms)}.\]
Therefore, if we define the operator $T$ by $T(f,g) = (T_1g, T_2f)$, then it maps $L^r(\oms) \times L^s(\oms)$ into itself.
In fact, \eqref{eq:reg0} indicates that $T$ is a contraction mapping on $L^r(\oms) \times L^s(\oms)$ provided $\delta > 0$ small enough.

Set
\[\mcq_1(x) = \int_{\oms} G(x,y) Q_1(y)\dy \quad \textup{and} \quad \mcq_2(x) = \int_{\oms} G(x,y) Q_2(y)\dy\]
for $x \in \oms$, which belong to $L^r(\oms) \times L^s(\oms)$ thanks to the condition $(Q_1, Q_2) \in L^{\sigma_1}(\oms) \times L^{\sigma_2}(\oms)$.
We also write \eqref{eq:reg1} in the operator form
\begin{equation}\label{eq:reg3}
(\psi,\phi) = T(\psi,\phi) + (\mcq_1,\mcq_2).
\end{equation}
Then, by invoking the Banach fixed-point theorem and the uniqueness of solutions to \eqref{eq:reg3}, we deduce that $(\psi, \phi) \in L^r(\oms) \times L^s(\oms)$. The proof is concluded.
\end{proof}

By \eqref{error est}, we have that $\|\mce_{\ep,(d,\tau)}\|_{X_{\ep}} \to 0$ as $\ep \to 0$. Thus, in view of \eqref{nonlin est}, all the hypotheses in Lemma \ref{lemma:reg} are fulfilled for system \eqref{eq:2} with the choice
\[\begin{cases}
\Omega_* = \Omega_\ep,\, (F_1,F_2) = (|\phi|^{p-1},|\psi|^{q-1}),\\
Q_1 = |\phi|^{p-1} P_1 + Q_{13} \in L^{\frac{p+1}{p-1}}(\Omega_\ep),\, Q_2 = |\psi|^{q-1} P_2 + O\(|\psi|^{q-2} \mone_{q>2}\) + Q_{23} \in
L^{\frac{q+1}{q-1}}(\Omega_\ep).
\end{cases}\]%
Consequently, there exists a small number $\zeta > 0$ such that
\[(\psi,\phi) \in L^{q+1+\zeta}(\Omega_\ep) \times L^{p+1+\zeta}(\Omega_\ep).\]
In particular, we see from \eqref{eq:2} that
\[-\Delta \psi \in L^{\frac{p+1+\zeta}{p}}(\Omega_\ep) \quad \textup{and} \quad -\Delta \phi \in L^{\frac{q+1+\zeta}{q}}(\Omega_\ep).\]
Combined with the Calder\'on-Zygmund estimate and the Sobolev embedding theorem, this yields 
\[\psi \in W^{2,\frac{p+1+\zeta}{p}}(\Omega_\ep) \subset L^{\frac{n(p+1+\zeta)}{np-2(p+1+\zeta)}}(\Omega_\ep) \quad \textup{and} \quad \phi \in W^{2,\frac{q+1+\zeta}{q}}(\Omega_\ep) \subset L^{\frac{n(q+1+\zeta)}{nq-2(q+1+\zeta)}}(\Omega_\ep)\]
provided $np > 2(p+1+\zeta)$ and $nq > 2(q+1+\zeta)$.\footnote{If $np \le 2(p+1+\zeta)$, then $\psi \in L^a(\Omega_\ep)$ for all $a \ge 1$. Similarly, if $nq \le 2(q+1+\zeta)$, then $\phi \in L^b(\Omega_\ep)$ for all $b \ge 1$.}
Besides, an elementary computation shows
\[\frac{n(p+1+\zeta)}{np-2(p+1+\zeta)} \ge q+1+(1+\eta)\zeta \quad \textup{and} \quad \frac{n(q+1+\zeta)}{nq-2(q+1+\zeta)} \ge p+1+(1+\eta)\zeta\]
for some small number $\eta > 0$ independent of $\zeta$. Hence
\[\psi \in L^{q+1+(1+\eta)\zeta}(\Omega_\ep) \quad \textup{and} \quad \psi \in L^{p+1+(1+\eta)\zeta}(\Omega_\ep).\]
Iterating the above process finitely many times, we deduce
\[(\psi,\phi) \in L^a(\Omega_\ep) \times L^b(\Omega_\ep) \quad \textup{for all } a, b \ge 1.\]
Finally, by feeding this information back to \eqref{eq:2}, we arrive at $(\psi,\phi) \in (L^{\infty}(\Omega_\ep))^2$ as desired.

\bigskip \noindent \textbf{Acknowledgement.}
The authors thank to Prof. Pistoia for suggesting and discussing the problem studied in this paper.
S. Jin was supported by the Basic Science Research Program through the National Research Foundation of Korea (NRF) funded by the Ministry of Science and ICT (NRF-2020R1A2C4002615).
S. Kim was supported by the Basic Science Research Program through the National Research Foundation of Korea (NRF) funded by the Ministry of Science and ICT (NRF2020R1C1C1A01010133, NRF2020R1A4A3079066).

\end{document}